\documentclass[12pt, reqno]{amsart}

 \usepackage{amsmath}
 \usepackage{amsfonts}
 \usepackage{amssymb}
 \usepackage{times}
 \usepackage{graphicx}
 \usepackage[justification=centering]{caption}
 \usepackage{subfigure}

\newtheorem{thm}{Theorem}[section]
\newtheorem{prop}[thm]{Proposition}
\newtheorem{cor}[thm]{Corollary}
\newtheorem{lem}[thm]{Lemma}

\theoremstyle{definition}

\newtheorem{rem}[thm]{Remark}
\numberwithin{equation}{section}

\newtheorem{fg*}{figure}

\newtheorem{thmc}{Conjecture }

\newcommand{\C}{{\mathbb C}}
\newcommand{\D}{{\mathbb D}}
\newcommand{\B}{{\mathbb B}}

\newcommand{\bn}{{\mathbb B}_n}

\newcommand{\cn}{\C^n}

\newcommand{\aut}{{\rm Aut}\,(\Omega)}

\begin{document}

\title[Evaluation functions and composition operators  ]
{Evaluation functions and composition operators on  Banach spaces of holomorphic functions}

\author{Guangfu Cao}
\address{Cao: School of Mathematics and Information Science,
Guangzhou University, Guangzhou 510006, China.}
\email{guangfucao@163.com}

\author{Li He*}
\address{He: School of Mathematics and Information Science,
Guangzhou University, Guangzhou 510006, China.}
\email{helichangsha1986@163.com}

\author[J. Li]{Ji Li}
\address{Li:  School of Mathematical and Physical Sciences, Macquarie University, NSW, 2109, Australia}
\email{ji.li@mq.edu.au}

\thanks{*Corresponding author, email: helichangsha1986@163.com}

\begin{abstract}
Let	$B(\Omega)$ be the Banach space  of holomorphic functions on a bounded connected domain  $\Omega$  in $\cn$, which  contains the ring of polynomials on $\Omega $.  In this paper  we first establish a criterion for $B(\Omega )$ to be reflexive
via evaluation functions on $B(\Omega )$, that is, $B(\Omega )$ is reflexive if and only if the evaluation functions  span the  dual spaces $(B(\Omega ))^{*} $.
Moreover, under suitable assumptions on $\Omega$ and $B(\Omega)$,
we establish a characterization of the composition operator $C_\varphi$ to be a Fredholm operator on $B(\Omega)$ via the property of
the holomorphic self-map $\varphi:\Omega\to\Omega$. Our new approach utilizes the symbols of composition operators to construct a linearly independent function sequence, which bypasses the use of boundary behavior of reproducing kernels as those may not be applicable in our general setting.

\end{abstract}

\maketitle
\section{Introduction}
\setcounter{equation}{0}
Let	$B(\Omega)$ be the Banach space  of holomorphic functions on a bounded connected domain  $\Omega$  in $\cn$, which  contains the ring of polynomials on $\Omega $. When   $  n>1$,    we may assume that $\Omega $ is simply connected since any holomorphic function on $\Omega $ can be analytically extended to the holes in $\Omega $ by Hartogs theorem.
Given a holomorphic self-map $\varphi:
\Omega\to\Omega$, the composition operator $C_\varphi: B(\Omega)\to B(\Omega)$
is defined by $$C_\varphi f=f\circ\varphi, \quad\forall f\in B(\Omega).$$
Moreover, let  $K_{w}$ denote the evaluation function at $w\in \Omega$, that is $$K_{w}(f)=f(w), \quad\forall f\in B(\Omega).$$

\subsection{Aims and background} The aim of  this paper is twofold: firstly, to build the criterion for reflexivity of the Banach space $B(\Omega)$ via the evaluation functions  on $B(\Omega)$; and secondly, to characterize the composition operator $C_\varphi$ to be a Fredholm operator on $B(\Omega)$. We now state the questions in details.

\newpage
\noindent\textbf{1. \it Reflexivity of $B(\Omega)$.}

Characterization of reflexivity of different types of Banach spaces is a fundamental problem in functional analysis. A very famous general criterion is the well-known Kakutani Theorem which utilizes the closed unit ball of the Banach space.

Turning to the specific Banach space $B(\Omega)$ that we focus in this paper, finding a specific criterion for reflexivity is certainly an important problem.

Our idea is to use evaluation functions on $B(\Omega)$, which plays the role of reproducing kernels on a Hilbert space. However, comparing to the reproducing kernels (on a Hilbert space), the property of evaluation functions on Banach spaces is far less known.
As we know,    an evaluation function can   be unbounded   even   in the Hilbert space of holomorphic functions. For instance, let $\mathbb{D}_{2}$ be the disc with radius  2 in the complex plane $\mathbb C$ and take $\{z^{n}\}$ as the orthogonal basis of the  space $H(\mathbb{D}_{2}).$ For $f, g\in H(\mathbb{D}_{2}) $, expand $f(z)=\sum_{n=0}^{\infty }a_{n}z^{n}, g(z)=\sum_{n=0}^{\infty }b_{n}z^{n}\in H(\mathbb{D}_{2})$ and then define the inner product of $f$ and $g$ as
\begin{align}\label{inner product}
\langle f, g\rangle=\sum_{n=0}^{\infty }a_{n}\overline{b_{n}}.
\end{align}
It is not difficult to verify that $H(\mathbb{D}_{2})$ is a Hilbert space. Take the point $z_{0}=\frac{3}{2},$ then $z_{0}\in \mathbb{D}_{2}$. Let $K_{z_{0}}$ be the evaluation function on $H(\mathbb{D}_{2})$ at $z_{0},$ that is,
$K_{z_{0}}(f)=f(z_{0})$ for any $f\in H(\mathbb{D}_{2}).$ We  see easily that $K_{z_{0}}$ is unbounded on $H(\mathbb{D}_{2}).$ In fact, if $f$ is taken as $f(z)=\sum_{n=0}^{\infty }\frac{1}{n}z^{n},$ then
$$
\|f\|=\sqrt{\sum_{n=0}^{\infty }\frac{1}{n^{2}}}<\infty .
$$
However, $K_{z_{0}}(f)=\sum_{n=0}^{\infty }\frac{1}{n}(\frac{3}{2})^{n}=\infty .$  Obviously, $H(\mathbb{D}_{2})$ is isomorphic to the weighted Hardy space $H^{2}(\beta )$ with weight $\beta_{n}=\frac{1}{2^{n}}.$

To  further understand  the  properties of evaluation functions on $B(\Omega)$,  we recall that in the Hilbert space of holomorphic functions, the reproducing kernel has a very important property  that  $\{K_ {z} | z \in U\}$ spans the full  space for any open subset or dense subset $U$ of the domain (see for example, \cite{CM}). The  similar conclusion  may not    hold  in the general Banach space of  holomorphic  functions.  Instead,  we consider the following question related to  the reflexivity of the specific Banach space
$B(\Omega)$. 

\smallskip
{\bf Question 1:} Whether we can characterize the reflexivity of $B(\Omega)$ via the evaluation functions on $B(\Omega)$?
\medskip

\noindent\textbf{2. \it Composition operator to be Fredholm on $B(\Omega)$.}

\smallskip

We first recall that a bounded linear operator $T$ on a Banach space is called a Fredholm operator if
$T$ has closed range, $\dim(\ker(T))<\infty$, and $\dim(\ker(T^*))<\infty$. It is
trivial that every invertible operator on a Banach space is Fredholm.

In   the Hilbert space $H(\Omega)$ of holomorphic functions on a bounded connected domain  $\Omega$  in $\cn$, it follows from the Riesz representation theorem for Hilbert spaces  that there exists a unique function $K_w\in H(\Omega)$ such that
$$
f(w)=\langle f, K_w\rangle,  \quad f\in  H(\Omega).
$$
The above function $K(z, w)=K_w(z)$ defined on $\Omega\times \Omega$, is  often called the reproducing kernel of $H(\Omega)$.
In \cite[Main Theorem]{CHZ2},    K. Zhu and  the first and second named authors of this paper  used   the theory of proper holomorphic maps to
show that
if
\begin{eqnarray}\label{e1.1}
	\|K_{w}\|=\sqrt{K(w,w)}\rightarrow \infty \hskip 4mm \mbox{as} \hskip 4mm w\rightarrow  \partial \Omega,
\end{eqnarray}
then the
following conditions are equivalent:
\begin{enumerate}
	\item[(i)] $C_\varphi$ is a Fredholm operator on $H(\Omega)$.
	\item[(ii)] $C_\varphi$ is an invertible operator on $H(\Omega)$.
	\item[(iii)] $\varphi\in\aut$.
\end{enumerate}

 In \cite{MZ}, J.S. Manhas and R. Zhao  pointed out  that      \cite[Main Theorem]{CHZ2}  requires an additional  assumption that  the sequence $\{f_{k}\}$ should converges to 0 weakly in $H,$ where
$$
f_{k}(w)=K(w,z_{k})/\sqrt{K(z_{k},z_{k})}=K(w,z_{k})/\|K_{z_{k}}\|.
$$
In fact,  Let  $\Omega =\{z\in \mathbb{C}:0<|z|<1\}$ and let $H$  denote the Hilbert space of holomorphic functions on $\Omega $  with orthonormal basis $e_n =z^n, n = -1,0,1, 2, 3,\ldots  .$ For any $f(z)=\sum_{n=-1}^{\infty }a_{n}z^{n}, g(z)=\sum_{n=-1}^{\infty }b_{n}z^{n}\in H,$ we define the inner product of $f$ and $g$ as in \eqref{inner product}.
Let
$\{z_{k}\}$ be a sequence of points in $\Omega $ such that $|z_{k}|\rightarrow 0$ as $k\rightarrow \infty .$
Then
$
f_{k}(w)=K(w,z_{k})/\sqrt{K(z_{k},z_{k})}=K(w,z_{k})/\|K_{z_{k}}\|
$
does not converges to 0 weakly in $H$ (see  \cite{MZ} for details). 

We point out that there exist Hilbert spaces of analytic functions whose
reproducing kernel does not satisfy the boundary behaviour \eqref{e1.1}. 
For example, the Hardy--Sobolev space $H_{\beta }^{2}(\D)$ on the unit disc $\D$ in $\mathbb C$ with $\beta >1/2$
has a bounded reproducing kernel $K(z,w)$. 
And hence, \cite[Main Theorem]{CHZ2} fails to cover this case.
In \cite{H},   the second named author of this paper  gave a new proof to show  that for  $n\geq 1, \beta >n/2$ and for $\varphi:{\B}_{n}\to{\B}_{n}$ to be a holomorphic self-map of ${\B}_{n}, $  $C_{\varphi}$ is Fredholm on $H_{\beta}^{2}$
if and only if $\varphi \in Aut(\mathbb{B}_{n})$, the automorphism group of ${\B}_{n}$. From this, we see that the boundary behavior \eqref{e1.1} is not a necessary condition for the Fredholmness of composition operator. 
Hence, the boundary behavior of reproducing kernel is not indeed necessary in more general situations.

The result in   \cite[Main Theorem]{CHZ2}
has several    applications    in the literature. For example, Fredholm composition operators on the Hardy space of the unit disk were characterized in \cite{B1, CS,CTW};  on the Bergman space of the unit disk were characterized in \cite{JS,B1}; on the Dirichlet space of the unit disk were characterized in \cite{C}. For Fredholm composition operators on other spaces of holomorphic functions on domains in $\mathbb{C} $ and $\mathbb{C}^n$, we refer to \cite{B1, C1, CTW,  CM, GGL, HO, KA, MB, MZ,Zo} and the references therein.

Thus, along the line of  \cite[Main Theorem]{CHZ2} on Hilbert space, the natural question that we consider is the following.

\smallskip
{\bf Question 2:} Whether we can establish a characterization of $C_\varphi$ to be a Fredholm operator on $B(\Omega)$ via 
the holomorphic self-map $\varphi:\Omega\to\Omega$?
\medskip

In this paper, we give confirmative answers to these two questions above.

\subsection{Statement of main results}
Recall again that $B(\Omega)$ is the Banach space  of holomorphic functions on a bounded connected domain  $\Omega$  in $\cn$, which  contains the ring of polynomials on $\Omega $.
We list the following  assumptions about
$\Omega$ and $B(\Omega)$ which will be used in stating our main results:

\begin{enumerate}
	\item[{\bf(a)}] For each $w\in\Omega$ and for every $f\in B(\Omega)$, the evaluation map $f\mapsto f(w)$ is a bounded
	linear functional on $B(\Omega)$. 
	
	\item[{\bf(b)}] Let $\aut$ denote the automorphisim group of $\Omega$. That is,
	$\varphi\in\aut$ if and only if $\varphi:\Omega\to\Omega$ is holomorphic,
	one-to-one, and onto. We assume that every $\varphi\in\aut$ induces a bounded
	composition operator on $B(\Omega)$. In this case, the operator $C_\varphi$ is invertible
	with $C_\varphi^{-1}=C_\psi$, where $\psi=\varphi^{-1}$.
	
	\item[{\bf(c)}] The points of $\Omega$ are separated by $B(\Omega)$ in the sense that for any finite
	sequence $\{a_k\}$ of distinct points in $\Omega$, the evaluation functions $\{K_{a_k}\}$
	are linearly independent.
\end{enumerate}

We now state our first main result as follows,
which answers {\bf Question 1} listed in Section 1.1.

\begin{thm} \label{th1.1}
Let $\Omega$ and $B(\Omega)$ satisfy Assumption {\bf(a)}.
Then $B(\Omega )$ is  reflexive  if and only if  for any open subset or dense subset $U$ of $\Omega , $
	$$
	\bigvee_{z\in U}\{K_{z}\}=(B(\Omega ))^{*},
	$$
	where $\bigvee\limits_{z\in U}\{K_{z}\}$ denotes the space spanned by $\{K_{z}\}_{z\in U}$.
\end{thm}

The proof of Theorem~\ref{th1.1} will be given in Section 2.

We note that the criterion established in Theorem~\ref{th1.1} can be applied to various settings. For example, taking $\Omega$ to be the unit ball $\mathbb{B}_{n} $ in $\mathbb C^n$, we see that our criterion applies to the 
  Hardy spaces,  Bergman spaces, Dirichlet spaces, Hardy--Sobolev spaces and   Bloch spaces on $\mathbb{B}_{n}$.

\smallskip

\smallskip

We now state our second main result, Theorem~\ref{th1.2} below,  which answers  {\bf Question 2} listed in Section 1.1.

\begin{thm} \label{th1.2}
	Let $\Omega$ and $B(\Omega)$ satisfy Assumptions {\bf(a)}, {\bf(b)} and {\bf(c)}.
	Moreover, we assume that $M_{z_{i}}\in \mathcal{M}(B(\Omega ))$ ($i=1,2,\cdots ,n$), the multiplier algebra on $B(\Omega )$ ,   and that  for any domain  $\Omega_{1} $ with $\Omega \subsetneqq \Omega_{1},$ there is at least an $h\in B(\Omega )$ such that $h$ can not be analytically extended to $\Omega_{1}.$ 
Let $\varphi:\Omega\to\Omega$ be a holomorphic self-map. Then the
following are equivalent:
	\begin{enumerate}
		\item[(i)] $C_\varphi$ is a Fredholm operator on $B(\Omega)$.
		\item[(ii)] $C_\varphi$ is an invertible operator on $B(\Omega)$.
		\item[(iii)] $\varphi\in\aut$.
	\end{enumerate}
\end{thm}

The proof of Theorem~\ref{th1.2} will be given in Section 4.

We would like to   mention that  the research on Fredholm composition operators dated back to  Cima, Thomson, Wogen and others in the 1970s in which   reproducing kernels played an essential role in the proof  on various holomorphic function spaces (see \cite{B1}, \cite{CF}, \cite{CHZ2}, \cite{CTW}, \cite{KA}, \cite{MB}). To prove our Theorem~\ref{th1.2}, we avoid the boundary behaviour of reproducing kernel as it may not be available, and use the symbols of composition operators to construct a linearly independent function sequence. 

Our Theorem~\ref{th1.2}  has many applications since
$B(\Omega)$ contains  various  Banach spaces of analytic functions including the important examples such as Hardy spaces, Bergman spaces, Dirichlet spaces, Bloch spaces and the Hardy--Sobolev spaces. 

We also point out that our Assumption {\bf (b)}, especially the condition that ``every $\varphi\in\aut$ induces a bounded
	composition operator on $B(\Omega)$'', is reasonable and necessary, since 
 there are   some Banach spaces which are not  invariant under automorphisms. For example, the weighted Hardy space $H^{2}(\beta) $ with $\beta(j) = 2^j,$ which is not invariant under automorphisms of the unit disk (see \cite{ CM}). Moreover, regarding the boundedness of composition operator, we would like to mention a recently result of B.Z. Hou and C.L. Jiang \cite{HJ}, where they gave a sufficient condition for the boundedness of composition operators with automorphism symbols on weighted Hardy spaces. They introduced the weighted Hardy space $H_{\beta }^2$ which is of polynomial growth, and proved that the composition operator with automorphism symbol is bounded on such spaces. In addition, they also illustrated that on a class of weighted Hardy space which is of intermediate growth, there exists an unbounded composition operator whose symbol is automorphism..

This paper is organized as follows. In Section 2, we provide the proof of Theorem~\ref{th1.1}. In Section 3 we give a criterion of a bounded linear operator to be a composition operator via evaluation function.   In Section 4, we provide the proof of Theorem~\ref{th1.2}. In Section 5, we give a characterization of a weighted composition operator to be a Fredholm operator on $B(\Omega)$.  In the last section we give some further remarks and the closely related conjectures.

\medskip

\section{ Properties of evaluation functions on $B(\Omega )$}
\setcounter{equation}{0}

In this section, we will discuss some  properties of evaluation functions. It can be proved that the closed unit ball of the Banach space of analytic functions is a normal family by the inner closed uniform boundedness of  the evaluation functions. This conclusion is somewhat unexpected, because it varies from the Hilbert spaces,  in general, the closed unit ball of an infinite dimensional Banach space is not weakly compact. In fact, Kakutani's Theorem, also called Eberlein--Shmuleyan's theorem, showed that the necessary and sufficient condition for the closed unit ball of a Banach space to be weakly compact is that the space is reflexive (see \cite{C0}).

We begin by recalling  a basic result in  \cite[p.5]{Ru}, followed by several auxiliary lemmas.
\begin{lem}\label{le2.1}
	If $\Lambda $ is a uniformly bounded family of holomorphic functions in $\Omega $, then $\Lambda $ is equicontinuous on every compact subset of $\Omega .$ In other words, $\Lambda$ is a normal family.
\end{lem}

\begin{lem}\label{le2.2}
	Let	$B(\Omega)$ be the Banach space  of holomorphic functions on a   domain  $\Omega$.  For any $z\in \Omega ,$ the evaluation function $K_{z}$ is bounded on $B(\Omega ).$ Then for any compact subset $K$ in $\Omega ,$
	$$
	\sup_{z\in K}\|K_{z}\|<\infty .
	$$
\end{lem}
\begin{proof}
	For any $f\in B(\Omega ),$ it is clear that
	$$
	\sup_{z\in K}|K_{z}(f)|=\sup_{z\in K}|f(z)|<\infty
	$$
	since $f$ is holomorphic on $\Omega .$ By the uniformly bounded principle, we see that
	$$
	\sup_{z\in K}\|K_{z}\|<\infty .
	$$
	The proof is complete.
\end{proof}

\begin{lem}\label{le2.3}	
	Let	$B(\Omega)$ be the Banach space  of holomorphic functions on a   domain  $\Omega$.  For any $z\in \Omega ,$ the evaluation function $K_{z}$ is bounded on $B(\Omega ).$   Then for arbitrary sequence $\{f_{k}\}$ in $(B(\Omega))_{1},$ the unit ball in $B(\Omega)$, we obtain that  $ \{f_{k}\}$ is a normal family.
\end{lem}
\begin{proof}
	Let $K$ be any compact subset of $\Omega .$ Choose another compact subset $\widetilde K\subset \Omega $ and an open subset $\widetilde \Omega\subset \Omega $ which satisfies
	$$
	K\subset \widetilde \Omega\subset \widetilde K,
	$$
	then
	\begin{eqnarray*}
		\sup_{k}\sup_{z\in \widetilde K}|f_{k}(z)|&=&\sup_{z\in \widetilde K}|K_{z}(f_{k})|\\
		&\leq &\sup_{k}\|f_{k}\|\cdot \sup_{z\in \widetilde K}\|K_{z}\| \\
		&\leq &\sup_{z\in \widetilde K}\|K_{z}\|<\infty
	\end{eqnarray*}
	by Lemma~\ref{le2.2}. Thus,
	$$
	\sup_{k}\sup_{z\in \widetilde \Omega}|f_{k}(z)|<\infty .
	$$
	This implies that $\{f_{k}\}$ is uniformly bounded on $ \widetilde \Omega,$ and also equicontinuous on every compact subset of $\widetilde \Omega $ by Lemma~\ref{le2.1}.  In particular $\{f_{k}\}$ is equicontinuous on $K.$ This completes the proof.
\end{proof}
It is well-known that for $1<p<\infty ,$ a sequence $\{f_{k}\}$ in Bergman space $A^{p}(\mathbb{D})$ convergences weakly to zero if and only if $\{\|f_{k}\|\}_{k}$ is bounded and $f_{k}(z)\rightarrow 0$ uniformly on compact subsets of $\mathbb{D}$ as $k\rightarrow \infty $ (see \cite{HKZ} Exercise 1.6.1). Since $A^{p}(\mathbb{D})$ $(1<p<\infty )$ is reflexive, we see that this conclusion is the consequence of  Kakutani's Theorem and our
Lemma~\ref{le2.3}.

\begin{prop}\label{prop2.4}
	Let	$B(\Omega)$ be the Banach space  of holomorphic functions on a   domain  $\Omega$.  For any $z\in \Omega ,$ the evaluation function $K_{z}$ is bounded on $B(\Omega ).$
	If $B(\Omega)$ is reflexive, $U$ is  an  open subset or a dense subset  of $\Omega $,   then
	$$
	\bigvee_{z\in U}\{K_{z}\}=(B(\Omega))^{*}.
	$$
\end{prop}
\begin{proof}
	Write $M=\bigvee_{z\in U}\{K_{z}\},$ if $M\neq (B(\Omega))^{*},$ then there is a non-zero $F\in (B(\Omega))^{**}$ such that $F|_M=0.$ Here $(B(\Omega ))^{**}$ is the quadratic dual space of $B(\Omega ).$ Since $B(\Omega)$ is reflexive, there is a nonzero function  $f\in B(\Omega)$ which satisfies $F=f^{**}\in (B(\Omega ))^{**},$ thus
	$$
	f(z)=K_{z}(f)=F(K_{z})=0.
	$$
	This implies that $f=0.$ Thus $F=f^{**}=0,$ which contradicts to $F\neq 0.$  This contradiction  completes the proof.
\end{proof}

In general,  if $B(\Omega )$  is not reflexive,  then  the set of evaluation functions may not span the dual space of $B(\Omega ).$ For example, it is well-known that $A^{1}(\mathbb{D}),$ the Bergman space on $\mathbb{D},$ is not reflexive, its dual space is the Bloch space $\mathfrak{B}$ (see \cite{HKZ}). The following proposition verifies our conclusion.

\begin{prop}\label{prop2.5}
	Suppose $A^{1}(\mathbb{D})$ is the Bergman space on the unit disc $\mathbb{D}$ in the complex plane $\mathbb{C}.$ Then for any $z\in \mathbb{D} ,$ the evaluation function $K_{z}$ is bounded on  $A^{1}(\mathbb{D}),$ and
	$$
	\bigvee_{z\in \mathbb{D} }\{K_{z}\}\neq  \mathfrak{B}.
	$$
\end{prop}
\begin{proof}
	It is obvious that $K_{z}$ is bounded for any $z\in \mathbb{D}.$ We prove this proposition by contradiction.
	Assume that
	$$
	\bigvee_{z\in \mathbb{D} }\{K_{z}\}= \mathfrak{B}.
	$$
	We are to prove that $(A^{1}(\mathbb{D}))_{1},$  the unit ball of $A^{1}(\mathbb{D}),$  is weakly compact, further $A^{1}(\mathbb{D})$ is reflexive by  Kakutani's Theorem.  Assume $\{f_{k}\}\subset (A^{1}(\mathbb{D}))_{1}.$ By Lemma 2.3,  there is a subsequence $\{f_{k_{j}}\}$ such that $f_{k_{j}}(z)\rightarrow  f(z)$ uniformly on compact subsets of $\mathbb{D} ,$ where $f$ is a holomorphic function on $\mathbb{D} .$  Thus for any $z\in \mathbb{D} ,$
	$$
	K_{z}(f_{k_{j}})=f_{k_{j}}(z)\rightarrow f(z).
	$$
	This implies that $ \{K_{z}(f_{k_{j}})\}$ is convergent. Further for any finite linear combination $\sum_{i=1}^{m}\alpha_{i}K_{z_{i}},$
	$\{ \sum_{i=1}^{m}\alpha_{i}K_{z_{i}}(f_{k_{j}})\}$ is also convergent. Since
	$$
	\bigvee_{z\in \mathbb{D} }\{K_{z}\}= \mathfrak{B},
	$$
	we see that  $\{F(f_{k_{j}})\}$ is convergent for any $F\in   \mathfrak{B}.$  Since $F$ is chosen arbitrarily, we know that there is a $G\in (A^{1}(\mathbb{D}))^{**}$ such that
	$f_{k_{j}}^{**}(F)\rightarrow G(F)$ for any $F\in  \mathfrak{B}$ since $ ((A^{1}(\mathbb{D}))^{**})_{1}$ is weakly star  compact. Write $f_{G}(z)=G(K_{z})$ for any $z\in \mathbb{D} ,$ we see that
	$$
	f_{G}(z)=G(K_{z})=\lim_{j\rightarrow \infty }f_{k_{j}}^{**}(K_{z})=\lim_{j\rightarrow \infty }f_{k_{j}}(z)=f(z).
	$$
	Hence, $f_{G}$ is a holomorphic function. Obviously, for any finite linear combination $P=\sum_{i=1}^{m}\alpha_{i}K_{z_{i}},$
	$$
	G(P)=\sum_{i=1}^{m}\alpha_{i}G(K_{z_{i}})
	=\sum_{i=1}^{m}\alpha_{i}f_{G}(z_{i}).
	$$
	For $r\in (0,1), $ write $f_{r}(z)=f_{G}(rz), $ then $f_{r}\in A^{1}(\mathbb{D}).$  Thus $P(f_{r})=\sum_{i=1}^{m}\alpha_{i}f_{r}(z_{i}).$
	Note
	$$
	\sum_{i=1}^{m}\alpha_{i}f_{r}(z_{i})\rightarrow \sum_{i=1}^{m}\alpha_{i}f_{G}(z_{i}) \hskip 5mm \mbox{as}\hskip 5mm r\rightarrow 1^{-},
	$$
	we may write $P(f_{G})=\sum_{i=1}^{m}\alpha_{i}f_{G}(z_{i}).$
	For any $F\in \mathfrak{B},$ assume
	$$
	P_{m}=\sum_{i=1}^{k_{m}}\alpha_{i}^{(m)}K_{z_{i}^{(m)}}\rightarrow F\hskip 4mm \mbox{in}\hskip 4mm  \mathfrak{B} {\rm\ \ as\ \ } m\to\infty,
	$$ then
	$$
	G(F)=\lim_{m\rightarrow \infty }G(P_{m})=\lim_{m\rightarrow \infty }P_{m}(f_{G}).
	$$
	Define
	$$
	F(f_{G})=\lim_{m\rightarrow \infty }P_{m}(f_{G})=G(F).
	$$
	Obviously, $F(f_{G})$ is well-defined, and $F(f_{k_{j}})\rightarrow  F(f_{G})$ as $j\to\infty$. Choose $F_{k}\in (\mathfrak{B})_{1}, $  the unit ball of $\mathfrak{B},$ such that
	$$
	| F_{k}(f_{r})|\geq \|f_{r}\|_{A^{1}(\mathbb{D})}-\frac{1}{k} \hskip 5mm \mbox{for any }\hskip 5mm k\in \mathbb{N}.
	$$
	Note
	$$
	|F_{k}(f_{r})|=\lim_{j\rightarrow \infty }|F_{k}(f_{k_{j}r})|\leq \sup_{j}\|F_{k}\|\cdot \|f_{k_{j}r}\|_{A^{1}(\mathbb{D})}\leq 1
	$$
	since $F_{k}\in (\mathfrak{B})_{1}, f_{k_{j}}\in (A^{1}(\mathbb{D}))_{1}.$
	We see that $\sup_{r\in (0,1)}\|f_{r}\|_{A^{1}(\mathbb{D})}\leq 1.$ This implies that $f_G\in A^{1}(\mathbb{D}).$ Thus $G=f_{G}^{**}.$ This shows that $(A^{1}(\mathbb{D}))_{1}$ is weakly compact. However, $A^{1}(\mathbb{D})$ is not reflexive, this contradicts to Kakutani's Theorem. We  complete the proof.
	\end{proof}

\begin{proof}[Proof of Theorem~\ref{th1.1}]
	We need only to prove the sufficiency by Proposition~\ref{prop2.4}.  Assume
	$$
	\bigvee_{z\in \Omega}\{K_{z}\}=(B(\Omega))^{*}.
	$$
	We will prove that $(B(\Omega ))_{1}$   is weakly compact, and further $B(\Omega )$ is reflexive by  Kakutani's Theorem.
Assume $\{f_{k}\}\subset (B(\Omega ))_{1}.$ By Lemma~\ref{le2.3},  there is a subsequence $\{f_{k_{j}}\}$ such that $f_{k_{j}}(z)\rightarrow  f(z)$ uniformly on compact subsets of $\Omega ,$ where $f$ is a holomorphic function on $\Omega.$  Similar to the proof of Proposition~\ref{prop2.5}, we may find a $G\in (B(\Omega))^{**}$ such that
	$f_{k_{j}}^{**}(F)\rightarrow G(F)$ for any $F\in  (B(\Omega ))^{*}$ since $ ((B(\Omega ))^{**})_{1}$ is weakly star  compact. Write $f_{G}(z)=G(K_{z})$ for any $z\in \Omega ,$ we see that
	$$
	f_{G}(z)=G(K_{z})=\lim_{j\rightarrow \infty }f_{k_{j}}^{**}(K_{z})=\lim_{j\rightarrow \infty }f_{k_{j}}(z)=f(z).
	$$
	Hence, $f_{G}$ is a holomorphic function. For any $F\in (B(\Omega ))^{*},$  $F(f_{G})$ is well-defined, and $F(f_{k_{j}})\rightarrow  F(f_{G}).$ By $
	\bigvee_{z\in \Omega}\{K_{z}\}=(B(\Omega))^{*}$ again, there exists a sequence $\{P_{m}\}\subset \bigvee_{z\in \Omega}\{K_{z}\}$ such that
$$
	P_{m}=\sum_{i=1}^{k_{m}}\alpha_{i}^{(m)}K_{z_{i}^{(m)}}\rightarrow F\hskip 4mm \mbox{in}\hskip 4mm  (B(\Omega))^{*} {\rm\ \ as\ \ } m\to\infty.
	$$
For arbitrary $\alpha \in \mathbb{C},$ since
\begin{eqnarray*}
F(\alpha f_{G})&=&\lim_{m\rightarrow \infty }P_{m}(\alpha f_{G})\\
&=&\alpha \lim_{m\rightarrow \infty }P_{m}( f_{G})\\
&=&\alpha \lim_{m\rightarrow \infty }G(P_{m})\\
&=&\alpha G(F)\\
&=&\alpha F(f_{G}),
\end{eqnarray*}
we see that  $F(\alpha f_{G})=\lim\limits_{m\rightarrow \infty }P_{m}(\alpha f_{G})$ is well-defined, and $F(\alpha f_{G})=\alpha F(f_{G}).$
Define 
	$$
	\|f_{G}\|:=\sup_{F\in ((B(\Omega ))^{*})_{1}}|F(f_{G})|.
	$$
Note that
$$
\lim_{j\rightarrow \infty }F(f_{k_{j}})= F(f_{G}),
$$
 and
$$
\|F(f_{k_{j}})\|\leq \|F\|\|f_{k_{j}}\|\leq 1,
$$
we get that $\|f_{G}\|<\infty .$ We now prove  $f_{G}\in B(\Omega )$ by contradiction. Suppose $f_{G}\notin B(\Omega ),$ then we define
$$
B=\bigvee\{B(\Omega ), f_{G}\},
$$
the space spanned by $B(\Omega )$ and $f_{G}.$ That is
$$
B=\{f+\alpha  f_{G}|f\in B(\Omega ),\alpha \in \mathbb{C}\}.
$$
For any $g=f+\alpha  f_{G},$ define the norm of $g$ as
$$
\|g\|_{1}=\|f\|+|\alpha |\| f_{G}\|.
$$
 Then $\|\cdot \|_{1}|_{B(\Omega )}=\|\cdot \|.$ It is easy to check that $B^{*}=(B(\Omega ))^{*}.$ In fact, for any $F\in (B(\Omega ))^{*},$ $F$ is well-defined on $B$ since $F( f_{G})$ is well-defined,  and
$$
|F(f+\alpha  f_{G})|=|F(f)+\alpha  F(f_{G})|\leq \|F\|(\|f\|+|\alpha |\|f_{G}\|)=\|F\|\|f+\alpha  f_{G}\|.
$$
Hence $F\in B^{*}.$ Conversely,  if $F\in B^{*},$ then for any $f\in B(\Omega ),$ $|F(f)|\leq \|F\|\|f\|_{1}=\|F\|\|f\|,$ this shows that $F\in (B(\Omega ))^{*}.$ By the assumption  $f_{G}\notin B(\Omega ),$ there is an $F_{0}\in B^{*}$ such that $F_{0}|_{B(\Omega )}=0,$ and $F_{0}(f_{G})\neq 0.$ However, we know that there is a sequence $\{f_{k}\}\subset B(\Omega )$ such that $F(f_{k})\rightarrow F(f_{G})$ for   any $F\in (B(\Omega ))^{*}.$ This implies that $F_{0}(f_{G})=0$ since  $B^{*}=(B(\Omega ))^{*}.$ This contradict shows that $f_{G}\in B(\Omega ).$
Thus $G=f_{G}^{**}.$ This shows that $(B(\Omega ))_{1}$ is weakly compact, and then $ B(\Omega )$ must be reflexive by Kakutani's Theorem.
	
	Now let $U$ be  any open subset or dense subset $U$ of $\Omega.$ We are to prove that
	$$
	\bigvee_{z\in U}\{K_{z}\}=(B(\Omega))^{*}
	$$
	if and only if
	$$
	\bigvee_{z\in \Omega}\{K_{z}\}=(B(\Omega))^{*}.
	$$
	We prove this argument by contradiction.
	Assume that  $$
	\bigvee_{z\in \Omega}\{K_{z}\}=(B(\Omega))^{*}
	$$
	and  there is an open subset or dense subset $U$ of $\Omega$ with
	$$
	\bigvee_{z\in U}\{K_{z}\}\neq (B(\Omega))^{*}.
	$$
	Then there is a non-zero $G\in (B(\Omega))^{**}$ such that $G|_{  \bigvee_{z\in U}\{K_{z}\}}=0.$
	Since $B(\Omega )$ is  reflexive, there is a non-zero $g\in B(\Omega)$ such that $G=g^{**}.$
	Thus for any $z\in U,$
	$$
	g(z)=g^{**}(K_{z})=G(K_{z})=0.
	$$
	Further $g=0$ on $\Omega$ by the uniqueness of the extension of holomorphic functions. The contradiction completes the proof.
\end{proof}
\begin{rem}It is well-known that the dual space of $H ^ {1}(\mathbb{B}_{n})$ is BMOA,  the holomorphic function space with bounded mean oscillation. According to Theorem 1.1, the evaluation functions on $H ^ {1}(\mathbb{B}_{n})$ cannot span  BMOA since $H ^ {1}(\mathbb{B}_{n})$ is not reflexive. Some other classical holomorphic function spaces, such as the Dirichlet space $\mathfrak{D}^{1}(\mathbb{D})$, the Hardy--Sobolev space $H_{\beta }^{1}(\mathbb{B}_{n}) (\beta \in \mathbb{R}),$  are also non reflexive, so the evaluation functions on them cannot span their dual spaces.
\end{rem}
\begin{prop}\label{prop2.8}
	Let	$B(\Omega)$ be the Banach space  of holomorphic functions on   $\Omega$.  For any $z\in \Omega ,$ the evaluation function $K_{z}$ is bounded on $B(\Omega ).$
	If  $\{z_{k}\}\subset \Omega $ is a sequence which convergences to  $w\in \Omega,$ then
	$$
	\|K_{z_{k}}-K_{w}\|\rightarrow 0 .
	$$
\end{prop}

\begin{proof}
	Assume the contrary, there is an $\epsilon_{0}>0$ and a sequence $\{f_{k}\}\subset (B(\Omega ))_{1}$ , the unit ball of $B(\Omega ),$ such that
	$$
	|(K_{z_{k}}-K_{w})(f_{k})|\geq \epsilon_{0}.
	$$
	That is
	$$
	|f_{k}(z_{k})-f_{k}(w)|\geq \epsilon_{0}.
	$$
	Write $F_{k}=f_{k}^{**}\in [(B(\Omega ))^{**}]_{1},$ where $(B(\Omega ))^{**}$ is the quadratic dual space of $B(\Omega ),$ $ [(B(\Omega ))^{**}]_{1}$ is the unit ball of $(B(\Omega ))^{**}.$ Then there is a subsequence $\{F_{k_{j}}\}$ such that $F_{k_{j}}$ converges weakly star to $F_{0}\in (B(\Omega ))^{**}.$ In particular,
	$$
	F_{k_{j}}(K_{z})\rightarrow F_{0}(K_{z}) \quad {\rm as\ } j\to\infty.
	$$
	On the other hand,
	$$
	F_{k_{j}}(K_{z})=K_{z}(f_{k_{j}})=f_{k_{j}}(z),
	$$
	hence there is a holomorphic function $f$ on $\Omega $ such that
	$$
	f_{k_{j}}(z)\rightarrow f(z)=F_{0}(K_{z})\quad {\rm as\ } j\to\infty.
	$$
	Assume $K\subset \Omega $ is any  compact subset of $\Omega ,$  then
	\begin{eqnarray*}
		\sup_{k_{j}}\sup_{z\in K}|f_{k_{j}}(z)|&\leq &\sup_{k_{j}}\|f_{k_{j}}\|\cdot \sup_{z\in K}\|K_{z}\|
		\leq\sup_{z\in K}\|K_{z}\| <\infty .
	\end{eqnarray*}
	By Lemma~\ref{le2.3}, we know that $\{f_{k_{j}}\}$ is a normal family. Further
	$$
	f_{k_{j}}|_{K}\rightarrow f|_{K}\hskip 5mm \mbox{uniformly}\hskip 5mm \mbox{as}\hskip 5mmj\rightarrow \infty .
	$$
	This contradicts to $|f_{k}(z_{k})-f_{k}(w)|\geq \epsilon_{0}.$ This means $\|K_{z_{k}}-K_{w}\|\rightarrow 0 .$
\end{proof}

Just as we mentioned  earlier,   J. S. Manhas and R. Zhao  \cite{MZ} set an counterexample, for  $\|K_{z}\|$, the norm of reproducing Kernel of the Hilbert space,   even if it tends to infinity, the normalized reproducing kernel $k_{z}=\frac{K_{z}}{\|K_{z}\|}$ may not converge weakly to 0. However, if the polynomials are dense in the space, this phenomenon won't occur.

\begin{prop}\label{prop2.9}
	Let	$B(\Omega)$ be the Banach space  of holomorphic functions on   $\Omega$.
	Assume that for any $z\in \Omega ,$   the evaluation function $K_{z}$ is bounded on $B(\Omega ) $ and $P[\Omega ]$, the ring of polynomials on $\Omega ,$ is dense in $B(\Omega ).$ For any sequence $\{z_{k}\}\subset \Omega $ and $\zeta \in \partial \Omega ,$ if  $\|K_{z_{k}}\|\rightarrow \infty $ as $z_{k}\rightarrow \zeta  ,$ then
	$$
	k_{z_{k}}=\frac{K_{z_{k}}}{\|K_{z_{k}}\|}\stackrel{weak^{*}}{\longrightarrow} 0  \hskip 4mm \mbox{as} \hskip 4mm k\rightarrow \infty .
	$$
\end{prop}

\begin{proof}
	Assume $P\in P[\Omega ],$ it is obvious that $K_{z_{k}}(P)\rightarrow P(\zeta )$ as $k\rightarrow \infty .$ Since $\|K_{z_{k}}\|\rightarrow \infty ,$ we see
	$$
	k_{z_{k}}(P)=\frac{K_{z_{k}}}{\|K_{z_{k}}\|}(P)=\frac{P(z_{k})}{\|K_{z_{k}}\|}\rightarrow 0 \quad {\rm as\ } k\to\infty.
	$$
	For any $f\in B(\Omega),$ take a sequence $\{P_{m}\}\subset P[\Omega ]$ such that $\|P_{m}-f\|\rightarrow 0$ as $m\rightarrow \infty .$ Then
	\begin{eqnarray*}
		|k_{z_{k}}(f)|&=&\bigg|\frac{K_{z_{k}}}{\|K_{z_{k}}\|}(f)\bigg|\\
		&\leq &\bigg|\frac{K_{z_{k}}}{\|K_{z_{k}}\|}(f)-\frac{K_{z_{k}}}{\|K_{z_{k}}\|}(P_{m})\bigg|+\bigg|\frac{K_{z_{k}}}{\|K_{z_{k}}\|}(P_{m})\bigg|\\
		&\leq & \bigg\|\frac{K_{z_{k}}}{\|K_{z_{k}}\|}\bigg\|\cdot \|f-P_{m}\|+\bigg|\frac{K_{z_{k}}}{\|K_{z_{k}}\|}(P_{m})\bigg|\\
		&=& \|f-P_{m}\|+\bigg|\frac{K_{z_{k}}}{\|K_{z_{k}}\|}(P_{m})\bigg|.
	\end{eqnarray*}
	This shows that $k_{z_{k}}(f)\rightarrow 0.$ That is
	$$
	k_{z_{k}}=\frac{K_{z_{k}}}{\|K_{z_{k}}\|}\stackrel{weak^{*}}{\longrightarrow} 0  \hskip 4mm \mbox{as} \hskip 4mm k\rightarrow \infty .
	$$
	The proof is complete.
\end{proof}

It should be noted that, similar to the reproducing kernel of the Hilbert space of analytic functions, the boundary behavior of the evaluation function on the Banach space of analytic functions depends on the structure of the space. For example, as we know, for $1\leq p<\infty ,$ if $\beta >n/p, $  then $\|K_{z}\|$, the norm  of the evaluation function $K_{z}$ on Hardy--Sobolev space $H_{\beta }^{p}$, is bounded on $ \mathbb{B}_{n} .$  As any function in  $H_{\beta }^{p} (\mathbb{B}_{n})$ is continuous on the boundary of $\mathbb{B}_{n},$  it can be seen that  when the  sequence $\{z_{k}\}$  in  $\mathbb{B}_{n} $ converges  to a point $\zeta \in \partial \mathbb{B}_{n},$  then according to the norm in $H_{\beta }^{p} (\mathbb{B}_{n}), $  $K_{z_{k}}$ converges  to $K_{\zeta },$ the evaluation function at $\zeta .$ Thus
$$
\|k_{z_{k}}-k_{\zeta }\|\rightarrow 0 \hskip 5mm \mbox{as}\hskip 4mm k\rightarrow \infty ,
$$
where $k_{z_{k}}=\frac{K_{z_{k}}}{\|K_{z_{k}}\|}, k_{\zeta }=\frac{K_{\zeta}}{\|K_{\zeta}\|}.$
If $\beta \leq n/p, $ then  $\|K_{z}\|\rightarrow \infty $  as $z\rightarrow \partial \mathbb{B}_{n} ,$ and
$$
k_{z}=\frac{K_{z}}{\|K_{z}\|}\stackrel{weak^{*}}{\longrightarrow} 0\hskip 5mm \mbox{as}\hskip 5mm  z\rightarrow  \partial \mathbb{B}_{n}.
$$

%$$$$$$$$$$$$$$$$$$$$$$$$$$$$$$$$$$$$$$$$$$$$$$$$$$$$$$$$$$$$$$$$$$$$$$$

\medskip

\section{ Characterization of composition operators on $B(\Omega)$}
\setcounter{equation}{0}

In this section we give the characterization of composition operators on $B(\Omega)$. First, we have the following result.

\begin{prop}\label{le3.1}
	Let	$B(\Omega)$ be the Banach space  of holomorphic functions on   $\Omega$. Assume that  the evaluation function $K_{z}$ is bounded on $B(\Omega ) $   for any $z\in \Omega.$
	Let $\varphi=(\varphi_{1},\ldots ,\varphi_{n}):\Omega\to\Omega$ be a holomorphic self-map such that
	$C_{\varphi}:B(\Omega)\to B(\Omega)$ is bounded. Then
	$$
	C_{\varphi}^{*}K_{z}=K_{\varphi (z)}.
	$$
\end{prop}

\begin{proof}
	For any $f\in B(\Omega ),$ we have
	$$
	C_{\varphi}^{*}K_{z}(f)=K_{z}(C_{\varphi }f)=K_{z}(f\circ \varphi)=f(\varphi (z))=K_{\varphi (z)}(f),
	$$
	and hence $C_{\varphi}^{*}K_{z}=K_{\varphi (z)}.$
\end{proof}

The equality  $C_{\varphi}^{*}K_{z}=K_{\varphi (z)}$ is the essential characterization of composition operators on  Banach spaces of holomorphic functions. In fact, for any bounded linear operator $T$ on the Banach space $B(\Omega ),$ if $TK_{z}$ is still an evaluation function for arbitrary $z\in \Omega,$ then $T$ must be a composition operator. Then we have the following

\begin{prop}\label{prop3.2}
	Let	$B(\Omega)$ be the Banach space  of holomorphic functions on   $\Omega$. Assume that  the evaluation function $K_{z}$ is bounded on $B(\Omega ) $   for any $z\in \Omega.$
	Let $T$ be a bounded linear operator on $B(\Omega).$ Then $T$ is a composition operator if and only if  $TK_{z}$ is also an  evaluation function on $B(\Omega)$ for any $z\in \Omega .$
\end{prop}

\begin{proof}
	Obviously, we need only to prove the sufficiency. Assume for any $z\in \Omega ,$  $T^{*}K_{z}$ is  an  evaluation function on $B(\Omega).$ Thus there is a $z_{T}\in \Omega $ such that
	$T^{*}K_{z}=K_{z_{T}}.$ Let $f_{i}(z)=z_{i}$ be the coordinate function on $\Omega ,$ then
	$$
	T^{*}K_{z}(f_{i})=K_{z_{T}}(f_{i})=f_{i}(z_{T})=z_{T,i},
	$$
	where $z_{T,i}$ is the $i$-th coordinate component of $z_{T}.$
	In addition,
	$$
	T^{*}K_{z}(f_{i})=K_{z}(Tf_{i})=Tf_{i}(z).
	$$
	Hence $Tf_{i}(z)=f_{i}(z_{T})=z_{T,i}.$ Set
	$$
	\varphi_{T}(z)=(\varphi_{T,1}(z),\ldots ,\varphi_{T,n}(z))=z_{T}=(z_{T,1},\ldots ,z_{T,n}),
	$$
	then
	$$
	\varphi_{T,i}(z)=z_{T,i}=f_{i}(z_{T})=Tf_{i}(z).
	$$
	This shows that $\varphi_{T,i}\in B(\Omega), $ and $\varphi_{T}$ is a holomorphic self-map on  $\Omega .$ For any $f\in  B(\Omega),$ we have
	\begin{eqnarray*}
		C_{\varphi_{T}}(f)(z)&=&f(\varphi_{T}(z))=f(z_{T})\\
		&=&K_{z_{T}}(f)=T^{*}K_{z}(f)\\
		&=&K_{z}(Tf)=Tf(z).
	\end{eqnarray*}
	This means that $T=C_{\varphi_{T}}.$
\end{proof}

%$$$$$$$$$$$$$$$$$$$$$$$$$$$$$$$$$$$$$$$$$$$$$$$$$$$$$$$$$$

\medskip

\medskip

\section{Proof of Theorem~\ref{th1.2}}
\setcounter{equation}{0}
To prove Theorem~\ref{th1.2}, we need   the following result.

\begin{lem}\label{le4.1}
Let	$B(\Omega)$ be the Banach space  of holomorphic functions on  $\Omega$. Assume that  the evaluation function $K_{z}$ is bounded on $B(\Omega ) $   for any $z\in \Omega.$
 Let $\varphi=(\varphi_{1},\ldots ,\varphi_{n}):\Omega\to\Omega$ be a holomorphic self-map. If $C_\varphi: B(\Omega)\to B(\Omega)$ is Fredholm, then $\varphi$ must be a univalent mapping.
\end{lem}

\begin{proof}
We first  prove that $\varphi $ must be an open map on $\Omega.$ Assume $U$ is an any open subset of $\Omega$, we need to prove that $\varphi(U)$ is also an open subset of   $\Omega ,$ which is equivalent to show that  $\varphi(z_{0})$ is the interior point of   $\varphi(U)$ for any  $z_{0}\in U.$   Since $C_{\varphi }$ is Fredholm, there are at most finite points $\{z_{i}\}_{i=1}^{m}, \{w_{i}\}_{i=1}^{m}$ such that
$\varphi(z_{i})=\varphi(w_{i}), i=1,\ldots ,m.$ Without loss of generality, assume $z_{0}\neq z_{i}, $ and $z_{0}\neq w_{i}.$ Write $\epsilon =dist(z_{0}, \{z_{i},w_{i}\}_{i=1}^{m}),$ then
$\varphi $ is univalent on $U_{\epsilon }=\{z| |z-z_{0}|<\frac{\epsilon}{2} \},$ further is univalent on $U\cap  U_{\epsilon }.$ Hence  $\varphi(U\cap  U_{\epsilon })$ is an open subset of ${\B}_n$. This shows that $ \varphi (z_{0})$ is the interior point of  $\varphi(U).$ That is,  $\varphi $ is an open map.

Next,   let us show  that  $\varphi $ is univalent on $\Omega.$  Otherwise,  there are two points  $z_{1}, z_{2}\in \Omega$ such that $ \varphi(z_{1})=\varphi(z_{2}).$  Thus there are two open subsets  $U_{1},$  $U_{2}$  of $\Omega$ which satisfies $z_{1}\in U_{1}, $ $z_{2}\in U_{2}$ and $U_{1}\cap U_{2}=\emptyset  $ such that  $\varphi(U_{1})\cap \varphi(U_{2})$ is a non-empty open subset of $\Omega$ since $\varphi$ is an open map.  Choose a sequence $\{w_{k}\}_{k=1}^{\infty }\subset \varphi(U_{1})\cap \varphi(U_{2})$
such that $w_{k}\neq w_{l}$  for $k\neq l,$ then there are $z_{k}\in U_{1}, \tilde{z}_{k}\in U_{2}$ such that
$$
\varphi(z_{k})=\varphi(\tilde{z}_{k})=w_{k}.
$$
Moreover,  $C_{\varphi }^{*}K_{z_{k}}=C_{\varphi }^{*}K_{\tilde{z}_{k}},$ that is  $C_{\varphi }^{*}(K_{z_{k}}-K_{\tilde{z}_{k}})=0.$ This contradicts to the Fredholmness of $C_{\varphi }$ by the independence of
$\{K_{z_{k}}-K_{\tilde{z}_{k}}\}.$ Hence $\varphi $ must be univalent. This proves Lemma~\ref{le4.1}.
\end{proof}

Now we use Lemma 4.1 to prove  Theorem~\ref{th1.2}.

\begin{proof}[Proof of  Theorem~\ref{th1.2}]
First we prove an equivalence between $(ii)$ and $(iii)
$. To do this,  it is enough to show that $(ii)\Rightarrow (iii).$ Assume $C_{\varphi }$ is invertible on $B(\Omega ),$ and $S$ is the inverse of $C_{\varphi }.$ That is, $SC_{\varphi }=C_{\varphi }S=I_{B(\Omega )},$ the identity element on $B(\Omega ).$ By Lemma~\ref{le4.1}, $\varphi $ is a univalent mapping, it is easy to see that $C_{\varphi }C_{\varphi^{-1} }=I_{B(\Omega )}=C_{\varphi }S.$ Hence $C_{\varphi^{-1} }=S$, thus $C_{\varphi^{-1} }$ is bounded on $B(\Omega ).$ Let $f_{i}(z)=z_{i},$ then $\varphi_{i}^{-1}(z)=C_{\varphi^{-1} }(f_{i}(z))\in B(\Omega ),$ where $\varphi_{i}^{-1}$ is the $i$th component of $\varphi^{-1}.$ This shows that $\varphi^{-1}$ is well-defined on $\Omega ,$ further $\varphi \in \aut .$

 Now let us prove an equivalence between $(i)$ and $(iii)$.
It suffices to prove  that $(i) \Rightarrow (iii).$
If $C_\varphi$ is a Fredholm operator on $B(\Omega)$, then $\varphi $ is univalent by Lemma~\ref{le4.1}. To see $\varphi\in\aut ,$ assume the contrary. Suppose  $ \Omega-\varphi(\Omega)\neq \phi .$ Then there are two possible  cases . \\
\emph{Case 1.} As shown in figure 1,  $\partial \Omega\subsetneqq \partial \varphi (\Omega)$ and $\partial \varphi (\Omega)-\partial \Omega$ contains a slice $L$ from some boundary points of $\Omega$
 into the inner of $\Omega$.

 As $\partial \varphi(\Omega)$ is a closed sub-manifold of $\overline{\Omega}$, $L$ must be an overlapping slice of two slice with opposite directions, we see that the pre-image of $L$ under $\varphi$ has two non-intersecting slices in $\partial\Omega$. Let $l_{1}$ and $l_{2}$ be two slices such that $\varphi(l_{1})=\varphi(l_{2})=L$ as a subset of $\partial \Omega$ but $\varphi(l_{1})$ and $\varphi(l_{2})$ have opposite directions in $\partial\varphi(\Omega).$
 \begin{figure}[h]
\centering
\includegraphics[height=50pt,width=120pt]{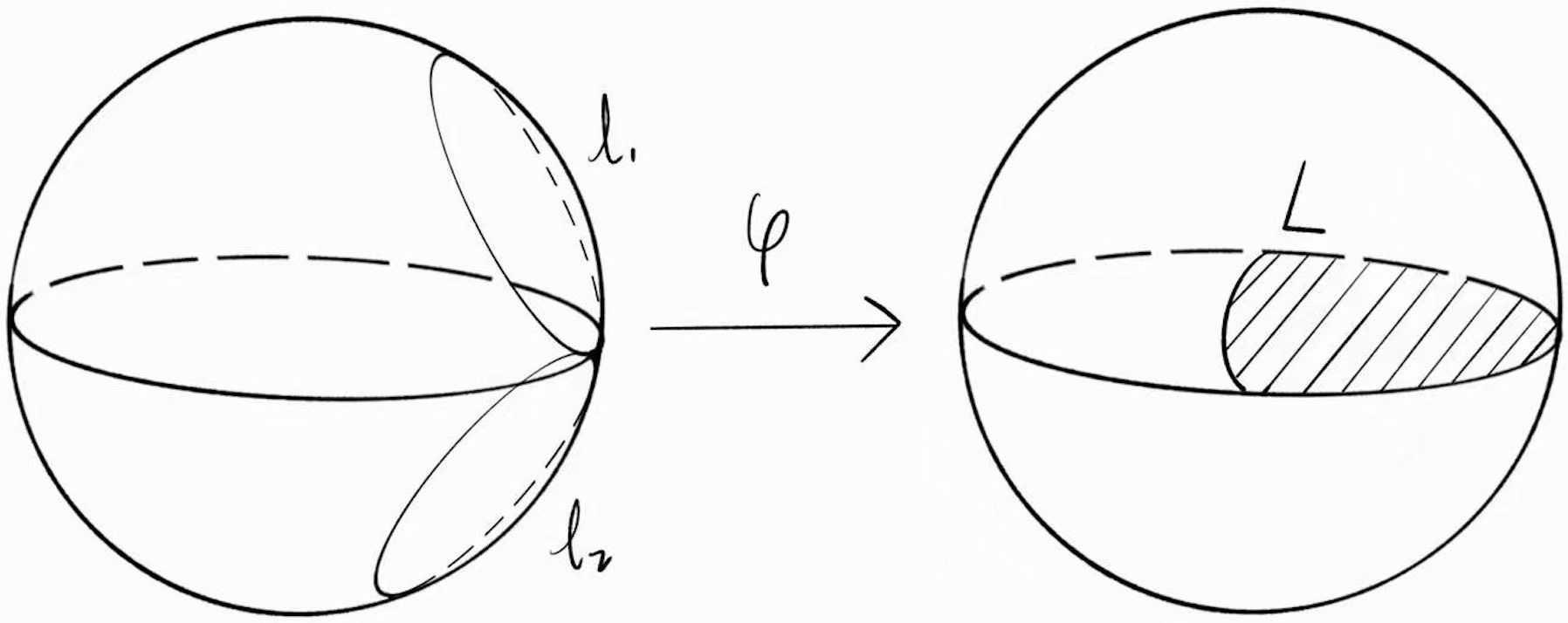}
\caption{}
\end{figure}\\
\emph{Case 2.} As shown in figure 2 and figure 3, there is a $z_{0}\in \Omega$ and an open neighborhood $U(z_{0})$ such that $U(z_{0})\cap \overline{\varphi(\Omega)}=\phi .$
\begin{figure}[h]
\centering
\includegraphics[height=50pt,width=120pt]{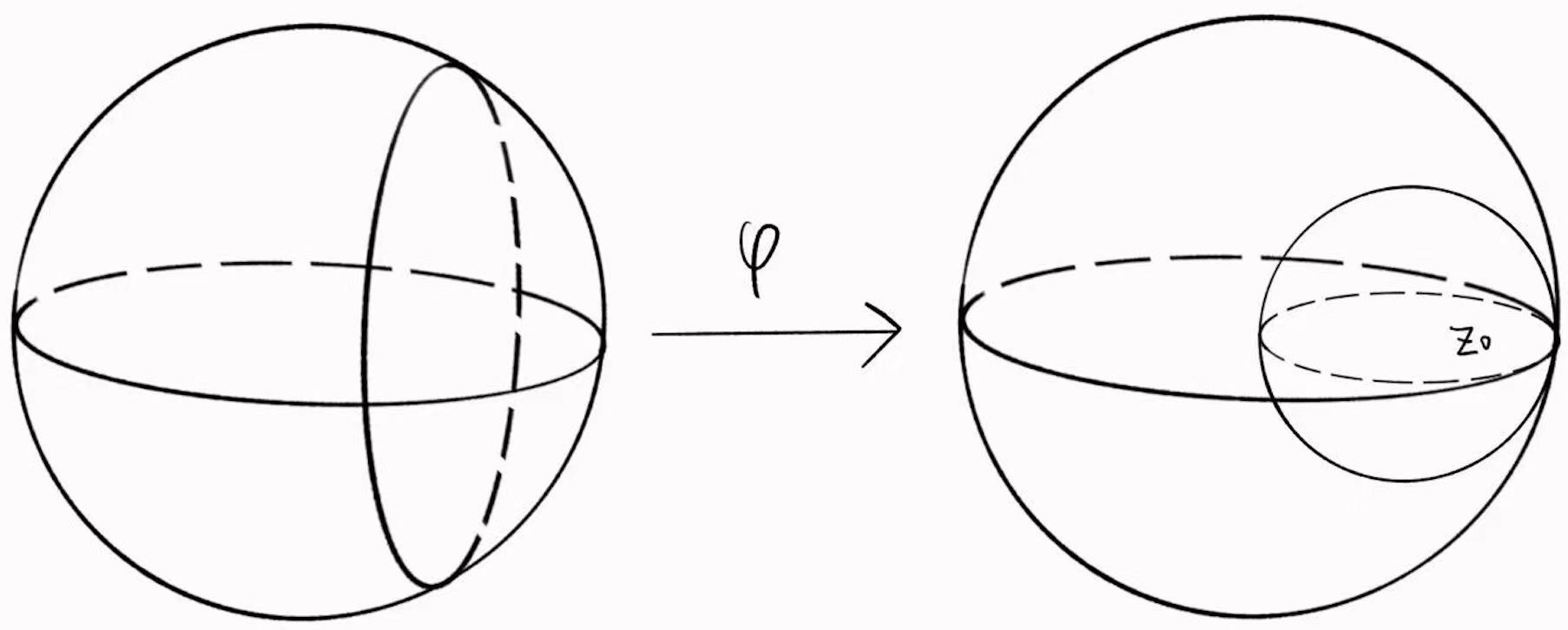}
\caption{}
\end{figure}\\
\begin{figure}[h]
\centering
\includegraphics[height=50pt,width=120pt]{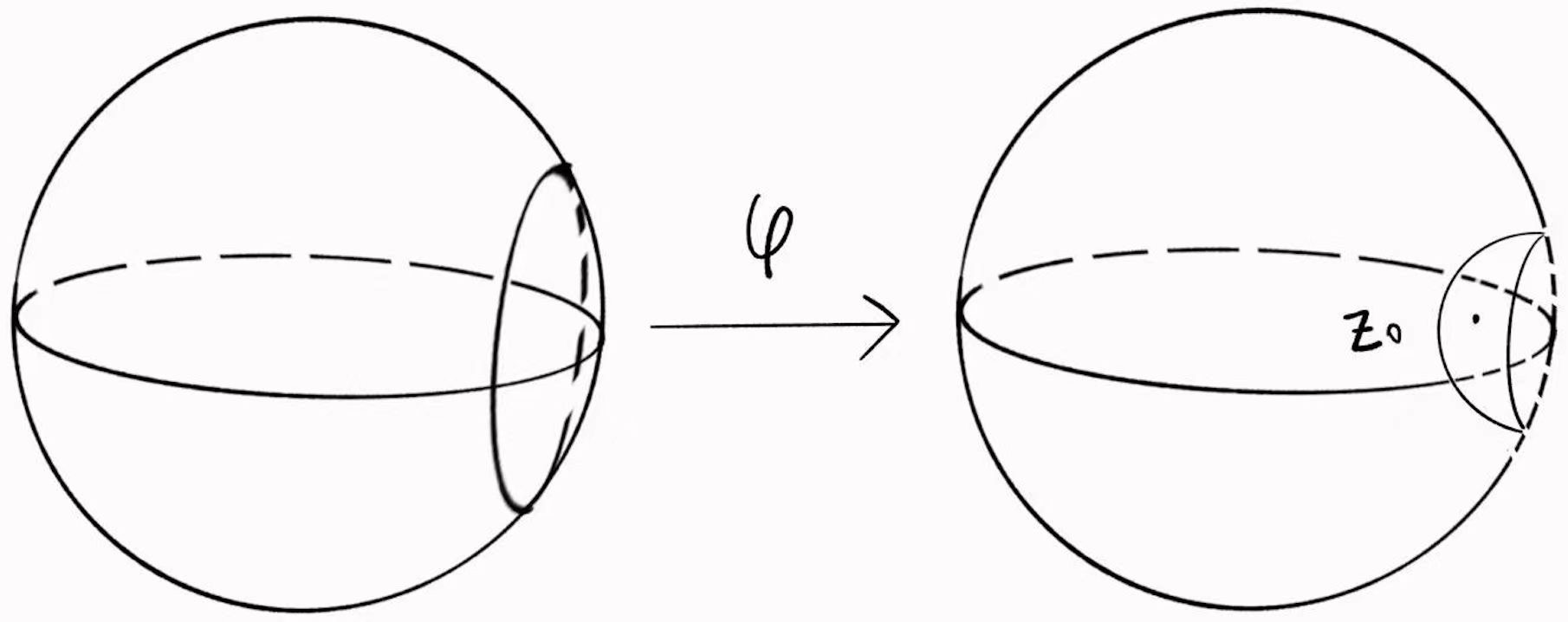}
\caption{}
\end{figure}\\

 Since $C_{\varphi }$ is Fredholm, we see that
$R(C_{\varphi }),$ the range of $C_{\varphi },$ is closed, and
$$
\dim [B(\Omega )/R(C_{\varphi })]=\dim\ker C_{\varphi}^{*}<\infty .
$$
Let $M$ be the subspace of $B(\Omega )$ such that $B(\Omega )=R(C_{\varphi })+M,$ and $\{h_{i}\}_{i=1}^{m}$ is the base of $M$. Then for any $f\in B(\Omega ),$ there are $\alpha_{i}\in \mathbb{C} ,i=1,\ldots, m,$ and $g\in B(\Omega )$ such that
$$
f =C_{\varphi }g+\sum_{i=1}^{m}\alpha_{i}h_{i}.
$$
In particular, for any $k\in \mathbb{N},$ there are $f_{i,k}\in B(\Omega )$ and $\alpha_{i,j}^{k}\in \mathbb{C}$ such that
$$
z_{i}^{k}=C_{\varphi }f_{i,k}+\sum_{j=1}^{m}\alpha_{i,j}^{k}h_{j},\hskip 5mm i=1,2,\ldots, n,
$$
thus
$$
(\varphi_{i}^{-1})^{k}=f_{i,k}|_{\varphi (\Omega )}+\sum_{j=1}^{m}\alpha_{i,j}^{k}C_{\varphi^{-1}}h_{j}, \hskip 5mm i=1,2,\ldots, n,
$$
 where $\varphi_{i}^{-1}$ is the $i$th component of $\varphi^{-1}.$ If $\varphi ^{-1}$ can not be analytically extended to  $\Omega .$ Set
$$
B=\bigvee\bigg\{B(\Omega )|_{\varphi(\Omega )},\{(\varphi_{i}^{-1})^{k}\}_{k=1}^{\infty },i=1,\ldots,n\bigg\},
$$
here $B$ is the  vector space spanned by  $B(\Omega )|_{\varphi(\Omega )}$ and $\{(\varphi_{i}^{-1})^{k}\}_{k=1}^{\infty }, i=1,2,\ldots ,n.$ Note  $\varphi^{-1}$ is defined on  $\varphi(\Omega ),$ if $\varphi^{-1}$ is not defined on $\Omega ,$  then
$$
\dim(B/B(\Omega )|_{\varphi(\Omega )})=\dim \bigvee_{k=1}^{\infty }\bigg\{{(\varphi_{i}^{-1})^{k}},i=1,2,\ldots,n\bigg\}=\infty .
$$
However,
$$
\dim(B/B(\Omega )|_{\varphi(\Omega )})<\infty
$$
by
$$
(\varphi_{i}^{-1})^{k}=f_{i,k}|_{\varphi (\Omega )}+\sum_{j=1}^{m}\alpha_{i,j}^{k}C_{\varphi^{-1}}h_{j},
$$
This contradiction shows that  $\varphi ^{-1}$ may be analytically extended to  $\Omega .$  Write $\widetilde{\varphi^{-1}}$ is the extension of $\varphi^{-1}.$ We are to prove that $\widetilde{\varphi^{-1}}(\Omega )= \Omega .$ Assume the contrary, that is, $ \Omega \subsetneqq \widetilde{\varphi^{-1}}(\Omega ).$
 Let $\Omega_{1}=\widetilde{\varphi^{-1}}(\Omega ).$
 By
$$
z_{i}^{k}=C_{\varphi }f_{i,k}+\sum_{j=1}^{m}\alpha_{i,j}^{k}h_{j},\hskip 5mm i=1,2,\ldots, n.
$$
we have
$$
(\widetilde{\varphi_{i}^{-1}})^{k}=f_{i,k}\circ \varphi \circ \widetilde{\varphi^{-1}}+\sum_{j=1}^{m}\alpha_{i,j}^{k}h_{j}\circ \widetilde{\varphi^{-1}},\hskip 5mm i=1,2,\ldots, n.
$$
Thus
$$
f_{i,k}\circ \varphi \circ \widetilde{\varphi^{-1}}=(\widetilde{\varphi_{i}^{-1}})^{k}-\sum_{j=1}^{m}\alpha_{i,j}^{k}h_{j}\circ \widetilde{\varphi^{-1}},\hskip 5mm i=1,2,\ldots, n.
$$
If $\varphi $ is undefined on $\Omega_{1}-\Omega ,$ then $f_{i,k}\circ \varphi \circ \widetilde{\varphi^{-1}}$ is undefined on $\Omega_{1}-\Omega .$
Set
$$
B=\bigvee_{k=1}^{\infty }\bigg\{f_{i,k}\circ \varphi \circ \widetilde{\varphi^{-1}},(\widetilde{\varphi_{i}^{-1}})^{k},i=1,\ldots,n\bigg\},
$$
$$
B_{0}=\bigvee_{k=1}^{\infty }\bigg\{(\widetilde{\varphi_{i}^{-1}})^{k},i=1,\ldots,n\bigg\},
$$
here $B$ is the  vector space spanned by  $\{f_{i,k}\circ \varphi \circ \widetilde{\varphi^{-1}},i=1,\ldots,n\}_{k=1}^{\infty }$ and $\{(\widetilde{\varphi_{i}^{-1}})^{k}, i=1,2,\ldots ,n\}_{k=1}^{\infty }$ and $B_{0}$  is the  vector space spanned by  $\{(\widetilde{\varphi_{i}^{-1}})^{k}, i=1,2,\ldots ,n\}_{k=1}^{\infty }.$ It is clear that
$$
\dim(B/B_{0})=\dim \bigvee_{k=1}^{\infty }\bigg\{f_{i,k}\circ \varphi \circ \widetilde{\varphi^{-1}},i=1,\ldots,n\bigg\}=\infty
$$
since  $ f_{i,k}\circ \varphi \circ \widetilde{\varphi^{-1}}$ is undefined on $\Omega_{1} -\Omega.$ However,
$$
\dim(B/B_{0})<\infty
$$
by
$$
f_{i,k}\circ \varphi \circ \widetilde{\varphi^{-1}}=(\widetilde{\varphi_{i}^{-1}})^{k}-\sum_{j=1}^{m}\alpha_{i,j}^{k}h_{j}\circ \widetilde{\varphi^{-1}},\hskip 5mm i=1,2,\ldots, n.
$$
This contradiction shows that $\varphi $ may be analytically extended to $\Omega_{1} .$ Let $\widetilde{\varphi }$ be the extension of $\varphi .$ Then
$$
\widetilde{\varphi }\circ \widetilde{\varphi^{-1}}|_{\varphi (\Omega )}=I_{\varphi (\Omega )},
$$
the identity on $\varphi(\Omega ).$ This shows that $ \widetilde{\varphi }\circ \widetilde{\varphi^{-1}}=I_{\Omega }$ by  the uniqueness of analytic extension. Thus $\widetilde{\varphi }$ is univalent on $\Omega_{1}$ and $ \widetilde{\varphi }(\Omega_{1})=\Omega .$  On the other hand,
$$
\widetilde{\varphi^{-1}}\circ \widetilde{\varphi}|_{\Omega }=\widetilde{\varphi^{-1}}\circ \varphi |_{\Omega }=I_{\Omega },
$$
the  identity on $\Omega .$ Hence $\widetilde{\varphi^{-1}}$ is univalent  on $\Omega .$ This shows that it  must be $\varphi (\Omega )=\Omega .$ In fact, if $\varphi(\Omega) \subsetneqq \Omega,$ then $\Omega \subsetneqq \Omega_{1}.$ Choose $h\in B(\Omega )$ such that $h$ can not be extended to $\Omega_{1},$ then for any $k\in \mathbb{N},$ there is an $f_{i,k}\in B(\Omega )$ such that
$$
z_{i}^{k}h=C_{\varphi }f_{i,k}+\sum_{j=1}^{m}\alpha_{i,j}^{k}h_{j},\hskip 5mm i=1,2,\ldots, n.
$$
Thus
$$
(\widetilde{\varphi_{i}^{-1}})^{k}h\circ\widetilde{\varphi^{-1}}=C_{\varphi }f_{i,k}\circ\widetilde{\varphi^{-1}}+\sum_{j=1}^{m}\alpha_{i,j}^{k}h_{j}\circ\widetilde{\varphi^{-1}},\hskip 5mm i=1,2,\ldots, n.
$$
Note $(\widetilde{\varphi_{i}^{-1}})^{k}h\circ\widetilde{\varphi^{-1}}$ is undefined on $\Omega -\varphi(\Omega)$ since $h$  can not be extended to $\Omega_{1}-\Omega.$ But $C_{\varphi }f_{i,k}\circ\widetilde{\varphi^{-1}}=f_{i,k}\circ \varphi \circ\widetilde{\varphi^{-1}}$ is well-defined on $\Omega $ by $\widetilde{\varphi }(\Omega_{1})=\Omega .$ Similar to the  discussion above, set
$$
B=\bigvee_{k=1}^{\infty }\bigg\{f_{i,k}\circ \varphi \circ \widetilde{\varphi^{-1}},(\widetilde{\varphi_{i}^{-1}})^{k}h\circ\widetilde{\varphi^{-1}} ,i=1,\ldots,n\bigg\},
$$
$$
B_{0}=\bigvee_{k=1}^{\infty }\bigg\{f_{i,k}\circ \varphi \circ \widetilde{\varphi^{-1}},i=1,\ldots,n\bigg\}.
$$
Obviously,
$$
\dim(B/B_{0})=\dim \bigvee_{k=1}^{\infty }\bigg\{(\widetilde{\varphi_{i}^{-1}})^{k}h\circ\widetilde{\varphi^{-1}} ,i=1,\ldots,n\bigg\}=\infty .
$$
However,
$$
\dim(B/B_{0})<\infty
$$
by
$$
(\widetilde{\varphi_{i}^{-1}})^{k}h\circ\widetilde{\varphi^{-1}}=C_{\varphi }f_{i,k}\circ\widetilde{\varphi^{-1}}+\sum_{j=1}^{m}\alpha_{i,j}^{k}h_{j}\circ\widetilde{\varphi^{-1}},\hskip 5mm i=1,2,\ldots, n.
$$
This contradiction shows that $\varphi (\Omega )=\Omega .$
Further $\varphi \in \aut .$
The proof of Theorem~\ref{th1.2} is complete.
\end{proof}

\begin{rem} \label{rem4.2} It is clear to see that
	
(i) Assumption {\bf (b)}   is indispensable  for Theorem~\ref{th1.2} since  the Banach space of holomorphic functions may not be invariant under automorphisms.

(ii)   Assumption {\bf(c)} is necessary for Theorem~\ref{th1.2} since there exists a  space such that it does not hold. For example, suppose $B(z)=\Pi_{i=1}^{\infty}\frac{\overline{z_{i}}}{|z_{i}|}\frac{z-z_{i}}{1-\overline{z_{i}}z}$ is the infinite Blaschke product on the unit disk $\mathbb{D}.$ Let
$$
H=\bigvee_{n=1}^{\infty }\{B^{n}\}
$$
with respect to the norm in $H^{\infty }.$ Obviously, for any $z\in  \mathbb{D},$ the evaluation function $K_{z}(f)=f(z)$ on $H$   is bounded.  For any zero point $z_{i}$ of $B(z),$ it is clear that $K_{z_{i}}=0$ on $H.$ They are, of course, linearly dependent.
As a consequence of Assumption {\bf(c)},    we have $K_w\not=0$ for every $w\in\Omega$.

\end{rem}

The following corollary  follows from Theorem~\ref{th1.2}.

\begin{cor}\label{cor4.2}
 Suppose $\Omega =\{z\in \mathbb{C}:0<|z|<1\}$ and let $H$  denote the Hilbert space of holomorphic functions on $\Omega $  with orthonormal basis $e_n =z^n, n = -1,0,,1, 2, 3,\cdots  .$ For any $f(z)=\sum_{n=-1}^{\infty }a_{n}z^{n}, g(z)=\sum_{n=-1}^{\infty }b_{n}z^{n}\in H,$ the inner product of $f$ and $g$ is defined as
$$\langle f, g\rangle =\sum_{n=-1}^{\infty }a_{n}\cdot \overline{b_{n}}.$$  Assume $\varphi : \Omega \rightarrow \Omega $ is a holomorphic self-map such that  $C_\varphi$ is  bounded on $H.$ Then  the following conditions are equivalent:
\begin{enumerate}
\item[(i)] $C_\varphi$ is a Fredholm operator on $H$.
\item[(ii)] $\varphi\in\aut$.
\end{enumerate}
\end{cor}

\begin{rem}\label{re4.3}
The methods in \cite{CHZ2} and \cite{MZ} are not applicable to this space, but it is not difficult to  calculate that the reproducing kernal of $H$ is
$$
K(z,w)=\frac{1}{\bar{z}w}+\frac{1}{1-w\bar{z}}.
$$
For any $z\in \Omega,$ $K_{z}(w)=K(z,w)\in H.$ Therefore, it satisfies the conditions of the Theorem 1.2. The above example  tells us that the boundary behavior of the reproducing kernel on Hilbert of  holomorphic functions is perhaps complex. The norm of the reproducing kernel
 $K(z,w)$ may be infinite or bounded when $z$ approaches to the boundary. It may also be infinite in some directions and finite in some other  directions. However,  no matter what the shape of the domain is, and no matter how complex the boundary behavior of the reproducing kernel or the evaluation function is, as long as the space contains the polynomial ring, similar conclusions are always true.

 The automorphic group of  $\Omega =\{z\in \mathbb{C}:0<|z|<1\}$ is perhaps well-known, but we didn't find the relevant  proof. For the completeness,  we give a proof here. If $\varphi \in \aut ,$ then $\varphi $ is a bounded holomorphic function on $\Omega .$ Let
 $$
 \varphi (z)=\sum_{n=1}^{\infty }b_{n}z^{-n}+\sum_{n=0}^{\infty }a_{n}z^{n}
 $$
 be the Laurent series of $\varphi .$
 Since $sup_{z\in \Omega }|\varphi (z)|\leq 1,$ it is routine  to check that $b_{n}=0, n=1,2,\cdots .$ Thus  $\varphi =\sum_{n=0}^{\infty }a_{n}z^{n}$ is well-defined on $\D,$ and it must be  bounded holomorphic  on $\D.$ Moreover, we claim that $a_{0}=0.$ Otherwise, if
  $a_{0}=\varphi(0)\neq 0,$ then  $|a_{0}|\leq 1$ since $|\varphi(z)|\leq 1.$ Note $\varphi$ is one-to-one, and $\varphi (\Omega )=\Omega,$ we see that $\varphi (z)\neq 0$ for any $z\in \Omega .$ If $0<|a_{0}|<1,$ then there is a $z_{0}\in \Omega $ such that $\varphi (z_{0})=a_{0}=\varphi (0), $ thus there is a sequence $\{w_{k}\}$ in a neighborhood $U$ of $0$ and  another sequence $\{z_{k}\}$ in a neighborhood $V$ of $z_{0}$ such that $\varphi(z_{k})=\varphi(w_{k}).$ This contradicts to the univalence of $\varphi .$ Hence $a_{0}=0$ or $|a_{0}|=1.$ However, $\varphi $ is holomorphic on $\D$ and $|\varphi (z)|\leq 1,$ it must be $a_{0}=0$ by maximum modulus principle.  This shows that $\varphi $ is an automorphic of $\D$ which satisfies $\varphi(0)=0.$ Thus there is a $\lambda \in \mathbb{T}$ such that $\varphi (z)=\lambda z.$  Furthermore, we know that  $\varphi : \Omega \rightarrow \Omega $ induces a Fredholm composition operator on $H$ if and only if $\varphi (z)=\lambda z$ for some $\lambda \in \mathbb{T}.$
\end{rem}

Note that if $\varphi $ is a proper holomorphic mapping on the domain $\Omega,$ then $\varphi $ must be a surjection. Therefore, a univalent proper holomorphic mapping must be automorphism. Then we can obtain another characterization for the Fredholm composition operator.
\begin{cor}
Let $\Omega$ and $B(\Omega)$ satisfy Assumptions {\bf(a)}, {\bf(b)} and {\bf(c)}. 
	Moreover, we assume that $M_{z_{i}}\in \mathcal{M}(B(\Omega ))(i=1,\ldots,n),$ the multiplier algebra on $B(\Omega )$,   and that for
any domain $\Omega_{1}$ with $\Omega \subsetneqq \Omega_{1},$ there is at least an $h\in B(\Omega )$  such that  $h$ can not be  analytically extended to  $\Omega_{1}.$ 
Let $\varphi:\Omega\to\Omega$ be a holomorphic self-map. Then $C_{\varphi }$ is a Fredholm operator if and only if $\varphi $ is a univalent proper holomorphic mapping.
\end{cor}

 \section{Weighted composition operators on $B(\Omega )$}
 Let $\Omega $ denote a bounded domain in $\mathbb{C}^{n}, $ $ H(\Omega )$ be the space of  holomorphic functions on $\Omega$. For a holomorphic function $\psi $ and a holomorphic self-map $\varphi $ on $\Omega ,$ the weighted composition operator $W_{\psi, \varphi }$ on $H(\Omega )$ is defined as
$$
W_{\psi ,\varphi }f = M_{\psi }C_{\varphi }f =\psi \cdot (f \circ \varphi ) ,\hskip 3mm f\in  H(\Omega ),
$$
where $M_{\psi }$ is called the  multiplication operator with symbol $\psi $.
\begin{lem}\label{le5.1}
Let	$B(\Omega)$ be the Banach space  of holomorphic functions on  $\Omega$. Assume that  the evaluation function $K_{z}$ is bounded on $B(\Omega ) $   for any $z\in \Omega.$  Let $\psi \in B(\Omega)$ such that $M_{\psi }$ is bounded on $B(\Omega).$ Then for any $z\in \Omega ,$
$$
M_{\psi }^{*}K_{z}=\psi(z)K_{z}.
$$
\end{lem}
\begin{proof}For any $f\in B(\Omega),$ we have
$$
M_{\psi }^{*}K_{z}(f)=K_{z}(M_{\psi }f)=K_{z}(\psi f)=\psi(z) f(z)=\psi(z)K_{z}(f).
$$
Hence,  $M_{\psi }^{*}K_{z}=\psi(z)K_{z}.$
\end{proof}
\begin{lem}
Let	$B(\Omega)$ be the Banach space  of holomorphic functions on  $\Omega$. Assume that  the evaluation function $K_{z}$ is bounded on $B(\Omega ) $   for any $z\in \Omega.$  Let $\psi \in B(\Omega)$ and $\varphi=(\varphi_{1},\cdots ,\varphi_{n}):\Omega\to\Omega$ be a holomorphic self-map such that both $M_{\psi }$ and $C_\varphi$ are Fredholm operators. If $M_{\psi }C_\varphi: B(\Omega)\to B(\Omega)$ is Fredholm, then
\begin{enumerate}
		\item[(i)] If $n=1,$ $\psi$ has at most finitely many  zeros in  $\Omega .$
		\item[(ii)]  If $n>1,$ $\psi$ has no zero points  in  $\Omega .$
		\item[(iii)]$\varphi$ must be  univalent.
	\end{enumerate}
\end{lem}
\begin{proof}For any $z\in \Omega ,$ we have
$$
(M_{\psi }C_{\varphi })^{*}K_{z}=C_{\varphi }^{*}M_{\psi }^{*}K_{z}=C_{\varphi }^{*}(\psi(z)K_{z})=\psi(z)K_{\varphi (z)}
$$
by Proposition 3.1 and Lemma 5.1.

If $n=1,$ and $\psi$ has infinite zeros in  $\Omega ,$ then  there is an infinite point set $Z_{\psi }=\{z|\psi(z)=0\}\subset \Omega $ such that
$$
(M_{\psi }C_{\varphi })^{*}K_{z}=0 \hskip 5mm \mbox{for any} \hskip 4mm z\in Z_{\psi }.
$$
This contradicts to the Fredholmness of $M_{\psi }C_\varphi .$ It completes the proof of (i).

If $n>1,$ and $Z_{\psi }\neq \phi ,$ then $Z_{\psi }$ must be an  infinite point set,  because in  high-dimensional case, the set of zeros of holomorphic functions cannot be the set of isolated points. Note that
$$
(M_{\psi }C_{\varphi })^{*}K_{z}=0 \hskip 5mm \mbox{for any} \hskip 4mm z\in Z_{\psi },
$$
and this contradicts to the Fredholmness of $M_{\psi }C_{\varphi }.$
The  proof of (ii) is finished.

To prove (iii), we first prove  that $\varphi $ is local univalent, that is, for any $z\in \Omega ,$ there is an open neighborhood $U(z)$ of $z$ such that $\varphi $ is univalent on $U(z).$ Otherwise, we may find two sequences $\{z_{k}\}, \{w_{k}\}\subset \Omega $ which converge to $z$ such that $\varphi (z_{k})=\varphi (w_{k}).$ Since $\psi $ has at most finite  zeros  in  $\Omega ,$ without loss of generality,  assume both $\psi (z_{k})$ and  $\psi (w_{k})$ is not zero for any $k.$ Thus
$$
(M_{\psi }C_{\varphi })^{*}\bigg(\frac{K_{z_{k}}}{\psi (z_{k})}-\frac{K_{w_{k}}}{\psi (w_{k})}\bigg)=K_{\varphi (z_{k})}-K_{\varphi (w_{k})}=0.
$$
Note $\Big\{\frac{K_{z_{k}}}{\psi (z_{k})}-\frac{K_{w_{k}}}{\psi (w_{k})}\Big\}_k$ is a linear independent sequence, this contradicts to the Fredholmness of $ M_{\psi }C_{\varphi }.$ Hence, $\varphi $ must be local univalent.  Furthermore, for any $z\in \Omega ,$ there is an open neighborhood $U(z)$ of $z$ such that $\varphi (U(z))$ is also an open subset of $\Omega .$ Similar to the proof of Lemma 4.1, we may further prove that $\varphi $ is a univalent map on $\Omega .$  We omit the detail here. The proof of (iii) is thus completed.
\end{proof}
\begin{thm}
Let $\Omega$ and $B(\Omega)$ satisfy Assumptions {\bf(a)}, {\bf(b)} and {\bf(c)}. 
	Moreover, we assume that $M_{z_{i}}\in \mathcal{M}(B(\Omega )) (i=1,\ldots,n),$ the multiplier algebra on $B(\Omega )$,   and that
for any domain $\Omega_{1} $ with $\Omega \subsetneqq \Omega_{1},$ there is at least an $h\in B(\Omega )$ such that $h$ can not be analytically extended to $\Omega_{1}.$ 
 Let $\varphi:\Omega\to\Omega$ be a holomorphic self-map. Then the
	following are equivalent:
\begin{enumerate}
		\item[(i)] $M_{\psi }C_\varphi: B(\Omega)\to B(\Omega)$ is Fredholm.
		\item[(ii)] Both $M_{\psi }$ and $C_\varphi $ are Fredholm.
\end{enumerate}
\end{thm}
\begin{proof}Obviously, if both  $M_{\psi }C_\varphi$ and $C_\varphi$ are Fredholm, then $M_{\psi }$ is also Fredholm.  To complete (i)$\Rightarrow $ (ii), we need only to prove that $C_\varphi$ is Fredholm under assuming  that $M_{\psi }C_\varphi$ is Fredholm. By Lemma 5.2, $\varphi $ is univalent, we are to prove that $\varphi $ is a surjection. Assume the contrary,
$\Omega -\varphi(\Omega )\neq \phi .$

If $\varphi^{-1} $ may be  analytically extended to $\Omega ,$ similar to the proof of Theorem 1.2,  we see easily that $C_{\varphi^{-1}}$ is the inverse of $C_{\varphi }$ and $\varphi \in \aut .$  Thus we may assume that $\varphi $ may not be analytically extended to $\Omega .$ The following proof is similar to that of Theorem 1.2, but for convenience of readers,  the detail of the proof is given here.

 Since $M_{\psi }C_{\varphi }$ is Fredholm, we see that
$R(M_{\psi }C_{\varphi }),$ the range of $C_{\varphi },$ is closed, and
$$
\dim [B(\Omega )/R(M_{\psi }C_{\varphi })]=\dim\ker (M_{\psi }C_{\varphi})^{*}<\infty .
$$
Let $M$ be the subspace of $B(\Omega )$ such that $B(\Omega )=R(M_{\psi }C_{\varphi })+M,$ and $\{h_{i}\}_{i=1}^{m}$ is the base of $M$, thus for any $f\in B(\Omega ),$ there are $\alpha_{i}\in \mathbb{C} ,i=1,\cdots m,$ and $g\in B(\Omega )$ such that
$$
f =M_{\psi }C_{\varphi }g+\sum_{i=1}^{m}\alpha_{i}h_{i}.
$$
In particular, for any $k\in \mathbb{N},$ there are $f_{i,k}\in B(\Omega )$ and $\alpha_{i,j}^{k}\in \mathbb{C}$ such that
$$
z_{i}^{k}=M_{\psi }C_{\varphi }f_{i,k}+\sum_{j=1}^{m}\alpha_{i,j}^{k}h_{j},\hskip 5mm i=1,2,\ldots, n.
$$
Thus
$$
(\varphi_{i}^{-1})^{k}=(\psi \circ\varphi^{-1})\cdot f_{i,k}|_{\varphi (\Omega )}+\sum_{j=1}^{m}\alpha_{i,j}^{k}h_{j}\circ\varphi^{-1}, \hskip 5mm i=1,2,\ldots, n,
$$
 where $\varphi_{i}^{-1}$ is the $i$th component of $\varphi^{-1}.$
 Further
$$
\frac{1}{\psi \circ  \varphi^{-1}}\cdot (\varphi_{i}^{-1})^{k}= f_{i,k}|_{\varphi (\Omega )}+\frac{1}{\psi \circ \varphi^{-1}}\cdot \sum_{j=1}^{m}\alpha_{i,j}^{k}h_{j}\circ\varphi^{-1}, \hskip 5mm i=1,2,\ldots, n
$$
on $\Omega -Z_{\psi \circ \varphi^{-1}}.$

Note $Z_{\psi  \circ \varphi^{-1}}$ is at most a finite  subset of $\Omega ,$ we see that $\{\frac{1}{\psi \circ  \varphi^{-1}}\cdot (\varphi_{i}^{-1})^{k}\}_{k=1}^{\infty }$ is a linearly independent sequence in $H(\Omega -Z_{\psi  \circ \varphi^{-1}}).$
Set
$$
B=\bigvee\Bigg\{B(\Omega )|_{\varphi(\Omega )},\bigg\{\frac{1}{\psi \circ  \varphi^{-1}}\cdot (\varphi_{i}^{-1})^{k}\bigg\}_{k=1}^{\infty },i=1,\ldots,n\Bigg\},
$$
here $B$ is the  vector space spanned by  $B(\Omega )|_{\varphi(\Omega )}$ and $\{\frac{1}{\psi \circ  \varphi^{-1}}\cdot (\varphi_{i}^{-1})^{k}\}_{k=1}^{\infty }, i=1,2,\ldots ,n.$ If $\varphi^{-1}$ can not be analytically extended to $\Omega ,$ then
$$
\dim(B/B(\Omega )|_{\varphi(\Omega )})=\dim \bigvee_{k=1}^{\infty }\bigg\{{\frac{1}{\psi \circ  \varphi^{-1}}\cdot (\varphi_{i}^{-1})^{k}},i=1,2,\ldots,n\bigg\}=\infty .
$$
However,
$
\dim(B/B(\Omega )|_{\varphi(\Omega )})<\infty
$
by
$$
\frac{1}{\psi \circ  \varphi^{-1}}\cdot (\varphi_{i}^{-1})^{k}= f_{i,k}|_{\varphi (\Omega )}+\frac{1}{\psi \circ \varphi^{-1}}\cdot \sum_{j=1}^{m}\alpha_{i,j}^{k}C_{\varphi^{-1}}h_{j}, \hskip 5mm i=1,2,\ldots, n.
$$
This contradiction shows that  $\varphi ^{-1}$ can be analytically extended onto $\Omega .$

Similar to the proof of Theorem 1.2, Write $\widetilde{\varphi^{-1}}$ is the analytic extension of $\varphi^{-1}.$ We are to prove that $\widetilde{\varphi^{-1}}(\Omega )= \Omega .$ Assume the contrary, that is, $ \Omega \subsetneqq \widetilde{\varphi^{-1}}(\Omega ).$ Let $\Omega_{1}=\widetilde{\varphi^{-1}}(\Omega ),$
 by
$$
z_{i}^{k}=M_{\psi }C_{\varphi }f_{i,k}+\sum_{j=1}^{m}\alpha_{i,j}^{k}h_{j},\hskip 5mm i=1,2,\ldots, n.
$$
we have
$$
(\widetilde{\varphi_{i}^{-1}})^{k}=(\psi\circ\widetilde{\varphi^{-1}})f_{i,k}\circ \varphi \circ \widetilde{\varphi^{-1}}+\sum_{j=1}^{m}\alpha_{i,j}^{k}h_{j}\circ \widetilde{\varphi^{-1}},\hskip 5mm i=1,2,\ldots, n.
$$
Thus
$$
(\psi\circ\widetilde{\varphi^{-1}})f_{i,k}\circ \varphi \circ \widetilde{\varphi^{-1}}=(\widetilde{\varphi_{i}^{-1}})^{k}-\sum_{j=1}^{m}\alpha_{i,j}^{k}h_{j}\circ \widetilde{\varphi^{-1}},\hskip 5mm i=1,2,\ldots, n.
$$
If $\varphi $ is undefined on $\Omega_{1}-\Omega ,$ then $(\psi\circ\widetilde{\varphi^{-1}})f_{i,k}\circ \varphi \circ \widetilde{\varphi^{-1}}$ is undefined on $\Omega_{1}-\Omega .$
Set
$$
B=\bigvee_{k=1}^{\infty }\bigg\{(\psi\circ\widetilde{\varphi^{-1}})f_{i,k}\circ \varphi \circ \widetilde{\varphi^{-1}},(\widetilde{\varphi_{i}^{-1}})^{k},i=1,\ldots,n\bigg\},
$$
$$
B_{0}=\bigvee_{k=1}^{\infty }\bigg\{(\widetilde{\varphi_{i}^{-1}})^{k},i=1,\ldots,n\bigg\},
$$
here $B$ is the  vector space spanned by  $\{(\psi\circ\widetilde{\varphi^{-1}})f_{i,k}\circ \varphi \circ \widetilde{\varphi^{-1}},i=1,\ldots,n\}_{k=1}^{\infty }$ and $\{(\widetilde{\varphi_{i}^{-1}})^{k}, i=1,2,\ldots ,n\}_{k=1}^{\infty }$ and $B_{0}$  is the  vector space spanned by  $\{(\widetilde{\varphi_{i}^{-1}})^{k}, i=1,2,\ldots ,n\}_{k=1}^{\infty }.$ It is clear that
$$
\dim(B/B_{0})=\dim \bigvee_{k=1}^{\infty }\bigg\{(\psi\circ\widetilde{\varphi^{-1}})f_{i,k}\circ \varphi \circ \widetilde{\varphi^{-1}},i=1,\ldots,n\bigg\}=\infty
$$
since  $(\psi\circ \widetilde{\varphi^{-1}} )f_{i,k}\circ \varphi \circ \widetilde{\varphi^{-1}}$ is undefined on $\Omega_{1} -\Omega.$ However,
$$
\dim(B/B_{0})<\infty
$$
by
$$
(\psi\circ\widetilde{\varphi^{-1}})f_{i,k}\circ \varphi \circ \widetilde{\varphi^{-1}}=(\widetilde{\varphi_{i}^{-1}})^{k}-\sum_{j=1}^{m}\alpha_{i,j}^{k}h_{j}\circ \widetilde{\varphi^{-1}},\hskip 5mm i=1,2,\ldots, n.
$$
This contradiction shows that $\varphi $ may be analytically extended to $\Omega_{1} .$ Let $\widetilde{\varphi }$ be the extension of $\varphi .$ Then
$$
\widetilde{\varphi }\circ \widetilde{\varphi^{-1}}|_{\varphi (\Omega )}=I_{\varphi (\Omega )},
$$
the identity on $\varphi(\Omega ).$ This shows that $ \widetilde{\varphi }\circ \widetilde{\varphi^{-1}}=I_{\Omega }$ by  the uniqueness of analytic extension. Thus $\widetilde{\varphi }$ is univalent on $\Omega_{1}$ and $ \widetilde{\varphi }(\Omega_{1})=\Omega .$  On the other hand,
$$
\widetilde{\varphi^{-1}}\circ \widetilde{\varphi}|_{\Omega }=\widetilde{\varphi^{-1}}\circ \varphi |_{\Omega }=I_{\Omega },
$$
the  identity on $\Omega .$ Hence $\widetilde{\varphi^{-1}}$ is univalent  on $\Omega .$ This shows that it  must be $\varphi (\Omega )=\Omega .$ In fact, if $\varphi(\Omega) \subsetneqq \Omega,$ then $\Omega \subsetneqq \Omega_{1}.$ Choose $h\in B(\Omega )$ such that $h$ can not be extended to $\Omega_{1}.$ then for any $k\in \mathbb{N},$ there is an $f_{i,k}\in B(\Omega )$ such that
$$
z_{i}^{k}h=M_{\psi }C_{\varphi }f_{i,k}+\sum_{j=1}^{m}\alpha_{i,j}^{k}h_{j},\hskip 5mm i=1,2,\ldots, n.
$$
Thus
$$
\frac{1}{\psi\circ\widetilde{\varphi^{-1}}}(\widetilde{\varphi_{i}^{-1}})^{k}h\circ\widetilde{\varphi^{-1}}=f_{i,k}\circ\varphi \circ\widetilde{\varphi^{-1}}+\frac{1}{\psi\circ\widetilde{\varphi^{-1}}}\sum_{j=1}^{m}\alpha_{i,j}^{k}h_{j}\circ\widetilde{\varphi^{-1}},
$$
$ i=1,2,\ldots, n.$
Note $\frac{1}{\psi\circ\widetilde{\varphi^{-1}}}(\widetilde{\varphi_{i}^{-1}})^{k}h\circ\widetilde{\varphi^{-1}}$ is undefined on $\Omega -\varphi(\Omega)$ since $h$  can not be extended to $\Omega_{1}-\Omega.$ But $f_{i,k}\circ\varphi \circ\widetilde{\varphi^{-1}}$ is well-defined on $\Omega $ by $\widetilde{\varphi }(\Omega_{1})=\Omega .$ Set
$$
B=\bigvee_{k=1}^{\infty }\bigg\{f_{i,k}\circ \varphi \circ \widetilde{\varphi^{-1}},\frac{1}{\psi\circ\widetilde{\varphi^{-1}}}(\widetilde{\varphi_{i}^{-1}})^{k}h\circ\widetilde{\varphi^{-1}} ,i=1,\ldots,n\bigg\},
$$
$$
B_{0}=\bigvee_{k=1}^{\infty }\bigg\{f_{i,k}\circ \varphi \circ \widetilde{\varphi^{-1}},i=1,\ldots,n\bigg\}.
$$
Obviously,
$$
\dim(B/B_{0})=\dim \bigvee_{k=1}^{\infty }\bigg\{\frac{1}{\psi\circ\widetilde{\varphi^{-1}}}(\widetilde{\varphi_{i}^{-1}})^{k}h\circ\widetilde{\varphi^{-1}} ,i=1,\ldots,n\bigg\}=\infty .
$$
However,
$$
\dim(B/B_{0})<\infty
$$
by
$$
\frac{1}{\psi\circ\widetilde{\varphi^{-1}}}(\widetilde{\varphi_{i}^{-1}})^{k}h\circ\widetilde{\varphi^{-1}}=f_{i,k}\circ\varphi \circ\widetilde{\varphi^{-1}}+\frac{1}{\psi\circ\widetilde{\varphi^{-1}}}\sum_{j=1}^{m}\alpha_{i,j}^{k}h_{j}\circ\widetilde{\varphi^{-1}},
$$
$i=1,2,\ldots, n.$
This contradiction shows that $\varphi (\Omega )=\Omega .$
Further $\varphi \in \aut .$
That is $C_{\varphi }$ is Fredholm. Consequently, $M_{\psi }$ is also Fredholm. The proof is thus  complete.
\end{proof}
\section{Further Remarks}
According to  Theorem 5.3,  if $M_{\psi }C_\varphi$ is a Fredholm operator on $B(\Omega ),$ then $C_{\varphi }$ is also a Fredholm operator,  so $\varphi \in \aut .$  However,  it is difficult to see the characteristics of the symbol $\psi $ according to the Fredholmness of $M_{\psi }.$ It can be seen from Lemma 5.1 that if $M_{\psi }$ is bounded on$B(\Omega ),$ then $\psi \in H^{\infty }(\Omega ),$ the algebra of bounded holomorphic functions on $\Omega ,$ and $\overline{\psi (\Omega )}\subset \sigma (M_{\psi }),$  the spectrum of $M_{\psi }.$ But we may wonder whether $\overline{\psi (\Omega )}= \sigma (M_{\psi })$ or not.  If
$$
k_{z}=\frac{K_{z}}{\|K_{z}\|}\stackrel{weak^{*}}{\longrightarrow} 0\hskip 5mm \mbox{as}\hskip 5mm  z\rightarrow  \partial \Omega ,
$$
then  we have
$$
\bigcap_{K\subset \Omega }\overline{\psi(\Omega -K)}\subset \sigma_{e} (M_{\psi }),
$$
where $ \sigma_{e} (M_{\psi })$ is the essential spectrum of  $M_{\psi },$ and  $K$ is the compact subset of $\Omega .$ But we do not know whether there is
$$
\bigcap_{K\subset \Omega }\overline{\psi(\Omega -K)}= \sigma_{e} (M_{\psi }).
$$

This problem involves another interesting problem, namely the reciprocal problem, which  was firstly raised by K.H. Zhu  in \cite{Z1} for the weighted Bergman space. He also proposed the following conjecture
\begin{thmc}[\cite{Z1}]\label{c1}
Let $0<p<\infty $ and $\alpha $ be real. If $f \in A^{p}_{\alpha }$ and there exists a constant $c > 0$ such that $|f(z)| \geq  c$ for all $z \in  \bn,$ then $1/f \in A^{p}_{\alpha }$.
\end{thmc}\label{c1}
The  case  when $\alpha $ is somewhere in the middle is difficult. We naturally propose a similar conjecture for $B(\Omega ):$
\begin{thmc}\label{c1}
Let	$B(\Omega)$ be the Banach space  of holomorphic functions on a   domain  $\Omega$.  If $\psi  \in B(\Omega)$ and there exists a constant $c > 0$ such that $|\psi (z)| \geq  c$ for all $z \in  \Omega ,$ then $1/\psi  \in B(\Omega).$
\end{thmc}\label{c1}
In \cite{CHZ}, the authors proved that if $f\in \mathcal{M}( A^{2}_{\alpha }),$  the algebra of multipliers on $A^{2}_{\alpha }$ (corresponding Hardy--Sobolev space),  then $\overline{f(\bn)}=\sigma(M_{f}).$ Furthermore, if there exists a constant $c > 0$ such that $|f(z)| \geq  c$ for all $z \in  \bn,$ then $1/f \in A^{2}_{\alpha }.$ Thus there is a weaker  conjecture
\begin{thmc}\label{c1}
Let	$B(\Omega)$ be the Banach space  of holomorphic functions on a   domain  $\Omega$.  If $\psi  \in \mathcal{M}(B(\Omega))$ and there exists a constant $c > 0$ such that $|\psi (z)| \geq  c$ for all $z \in  \Omega ,$ then $1/\psi  \in B(\Omega).$
\end{thmc}\label{c1}
Obviously, the conjecture is equivalent to  $\overline{\psi (\Omega )}= \sigma (M_{\psi }).$

\bigskip
\bigskip
 \noindent
{\bf Acknowledgments:} The authors would like to thank Prof. L. Yan and Prof. K. Zhu for helpful discussions.
 G.F. Cao was supported by NNSF of China (Grant Number 12071155).
L. He  was supported by NNSF of China (Grant Number  11871170). J. Li is supported by the Australian Research Council (ARC) through the
research grant DP220100285.

\end{document}